\documentclass[a4paper,11pt]{article}

\usepackage{amsmath,amsthm,amsfonts,amssymb,stmaryrd,amscd}
\usepackage{imakeidx}
\usepackage{graphicx}

\usepackage[francais, english]{babel}
\usepackage[T1]{fontenc} 
\usepackage[utf8]{inputenc}
\usepackage[babel=true]{csquotes}

\usepackage{pifont}
\usepackage{mathrsfs}

\input xy
\xyoption{all}

\usepackage[dvipsnames]{xcolor}
\usepackage{hyperref}
\makeindex[title=Index of Notation]

\usepackage[refpage]{nomencl}
\makenomenclature

\usepackage{ifthen}
  \renewcommand{\nomgroup}[1]{%
  \item[\bfseries
  \ifthenelse{\equal{#1}{G}}{Groupoids, deformation and blowup spaces}{%
  \ifthenelse{\equal{#1}{B}}{Fiber bundles}{%
   \ifthenelse{\equal{#1}{H}}{$\mathbf{C^*}$-Algebras}{%
  \ifthenelse{\equal{#1}{I}}{$\mathbf{KK}$-elements}{%
   \ifstrequal{#1}{O}{Other Symbols}{}}}}}%
    ]}

\newcommand{\titre}{Titre}

\numberwithin{equation}{section}

{\theoremstyle{definition}
\newtheorem{definition}{Definition}[section]
\newtheorem*{defmc}{\titre}

\newtheorem{remark}[definition]{Remark}
\newtheorem{remarks}[definition]{Remarks}

\newtheorem{examples}[definition]{Examples}
\newtheorem{question}{Question}
}
\newtheorem{proposition}[definition]{Proposition}
\newtheorem{proposition-definition}[definition]{Proposition-Definition}

\newtheorem{theorem}[definition]{Theorem}

\newtheorem{corollary}[definition]{Corollary}

\newenvironment{dmc}[1]
     {\renewcommand {\titre} {#1}\begin{defmc}}
     {\end{defmc}}

\newcommand{\cB}{\mathcal{B}}

\newcommand{\cF}{\mathcal{F}}
\newcommand{\cG}{\mathcal{G}}
\newcommand{\cJ}{\mathcal{J}}

\newcommand{\cR}{\mathcal{R}}

\newcommand{\cD}{\mathcal{D}}
\newcommand{\cK}{\mathcal{K}}
\newcommand{\cA}{\mathcal{A}}
\newcommand{\cC}{\mathcal{C}}
\newcommand{\cN}{\mathcal{N}}
\newcommand{\cI}{\mathcal{I}}
\newcommand{\cP}{\mathcal{P}}

\newcommand{\cH}{\mathcal{H}}
\newcommand{\cM}{\mathcal{M}}

\newcommand{\cS}{\mathcal{S}}

\newcommand{\R}{\mathbb{R}}

\newcommand{\bS}{\mathbb{S}}
\newcommand{\N}{\mathbb{N}}
\newcommand{\C}{\mathbb{C}}
\newcommand{\G}{\mathbb{G}}

\newcommand{\Z}{\mathbb{Z}}

\newcommand{\gA}{\mathfrak{A}}
\newcommand{\gN}{\mathfrak{N}}
\newcommand{\gG}{\mathfrak{G}}
\newcommand{\gH}{\mathfrak{H}}

\newcommand{\ronde}{\mathaccent'027}

\newcommand{\ev}{{\rm ev}}

\newcommand{\rra}{\rightrightarrows}
\newcommand{\lra}{\longrightarrow}

\newcommand{\resp}{{\it resp.}\/ }

\newcommand{\ind}{{{\mathrm{ind}}}}

\newcommand{\SBlup}{SBlup}

\newcommand{\DNC}{DNC}

\newcommand{\ie}{{\it i.e.}\/ }
\newcommand{\eg}{{\it e.g.}\/ }
\newcommand{\cf}{{\it cf.}\/ }

\reversemarginpar

\usepackage{color}

\everymath{\displaystyle}

\begin{document}

\begin{center}
{\renewcommand{\thefootnote}{\*}
{\Large\bf  Lie groupoids, pseudodifferential calculus and index theory}\footnote{{ The authors were partially supported by ANR-14-CE25-0012-01 (SINGSTAR).}}
\setcounter{footnote}{0}

\bigskip

{\sc by Claire Debord and Georges Skandalis}

}\end{center}

{\footnotesize
\vskip 2pt Universit\'e Paris Diderot, Sorbonne Paris Cit\'e
\vskip-2pt  Sorbonne Universit\'es, UPMC Paris 06, CNRS, IMJ-PRG
\vskip-2pt  UFR de Math\'ematiques, {\sc CP} {\bf 7012} - B\^atiment Sophie Germain 
\vskip-2pt  5 rue Thomas Mann, 75205 Paris CEDEX 13, France

\vskip-2pt cdebord@math.univ-paris-diderot.fr
\vskip-2pt skandalis@math.univ-paris-diderot.fr
}

\vspace{1cm}

 \begin{abstract} 
 Alain Connes introduced the use of Lie groupoids in noncommutative geometry in his pioneering work on the  index theory of foliations. In the present paper, we recall the basic notion involved: groupoids, their $C^*$-algebras, their pseudodifferential calculus... We review several recent and older advances on the involvement of Lie groupoids in noncommutative geometry. We then propose some open questions and possible developments of the subject.

 \end{abstract}

\tableofcontents

\renewcommand\theenumi{\alph{enumi}}
\renewcommand\labelenumi{\rm {\theenumi})}
\renewcommand\theenumii{\roman{enumii}}

\section{Introduction}

Groupoids, and especially smooth ones, appear naturally in various areas of modern mathematics. One can find a recent overview on Lie groupoids, with a historical introduction in \cite{Cartier}.

\bigskip 
Our aim in this paper is to review some advances in the study of Lie groupoids as objects of non commutative geometry. This theory of Lie groupoids is very much linked with various index problems. A main tool for this index theory is the corresponding pseudodifferential calculus. On the other hand, index theory and pseudodifferential calculus is strongly linked with deformation groupoids.

\bigskip 
Groupoids first appeared in the theory of operator algebras in the measurable - von Neumann algebra - setting. They were natural generalizations of actions of groups on spaces.  These crossed product operator algebras go back to the \enquote{group measure space construction} of Murray von Neumann (\cite{MurvNeu}), who used it in order to construct factors of all different types. It often happens that two group actions give rise to isomorphic groupoids (especially in the world of measurable actions - \ie the von Neumann algebra case). It was noticed in \cite{FelMoo1} that the corresponding operator algebras only depend on the corresponding groupoid. A recent survey of this measurable point of view can be found in \cite{Gabo}.

\bigskip
Two almost simultaneous major contributions forced groupoids into the topological noncommutative world \ie $C^*$-algebras:\begin{itemize}
\item The construction by Jean Renault of the $C^*$-algebra of a locally compact groupoid, its representation theory (\cite{Ren})...
\item The construction by Alain Connes of the von Neumann algebra and the $C^*$-algebra of a foliation based on its holonomy groupoid (\cite{ConnesLNM, ConnesSurvey, ConnesNCG}). 
\end{itemize}

Moreover, as a motivation for Connes was the generalization of the Atiyah-Singer index theorem (\cite{AtSing1, AtSing4}), he used the smooth longitudinal structure in order to construct the associated pseudodifferential calculus and the $C^*$-algebraic exact sequence of pseudodifferential operators. This allowed him to construct the analytic index, and to prove a measured index theorem (\cite{ConnesLNM}). Very soon after, he constructed a topological index with values in the $C^*$-algebra of the foliation (\cite{ConnesSurvey}). The corresponding index theorem was proved in \cite{CoSk}.

This pseudodifferential calculus on groupoids and the construction of the analytic index in the groupoid $C^*$-algebra was then easily generalized to all Lie groupoids (\cite{MonthPie, NWX}). They gave rise to several index theorems - \cf \cite{DLN, CRLM, DebLescGroupoids}.

\bigskip Another very nice construction of Connes (see \cite{ConnesNCG}) gave a geometric insight on this generalized index theory: the construction of the \emph{tangent groupoid}. This tangent groupoid allowed to construct the analytic index of (pseudo)differential operators without pseudodifferential calculus. It was also used to give beautiful alternate proofs of the Atiyah-Singer index theorem (\cite{ConnesNCG, DLN}).

\medskip Connes' tangent groupoid was an inspiration for many papers (\cf \cite{HilsSkMorph, MonthPie, NWX}...) where this idea was  generalized to various geometric contexts. Its natural setting is the \emph{deformation to the normal cone} (DNC) construction. Since DNC is functorial, Connes' construction can be extended to any case of a sub-Lie groupoid of a Lie groupoid (see \cite{DS5}). Moreover, this construction is immediately related to the Connes-Higson $E$-theory (\cf \cite{ConHig}). It therefore opened a whole world of deformation groupoids that are useful in many situations and gave rise to many interesting $K$-theoretic constructions and computations.

\medskip One also sees that the $C^*$-algebra extension of the pseudodifferential operators on a groupoid is directly related to the one naturally associated with the DNC construction and the canonical action of $\R_+^*$ on it (\cite{AMMS, DS1}). There is a well defined Morita equivalence between these exact sequences, and the corresponding bimodule gives an alternative definition of the pseudodifferential calculus on a groupoid (\cf \cite{DS1}) - which in turn should be used to various contexts. 

\bigskip In the present survey, we recall definitions and several examples of Lie groupoids and describe their $C^*$-algebras. Next, we study the pseudodifferential operators associated with Lie groupoids from various view points.  Examples of various constructions of groupoids giving rise to interesting $K$-theoretic computations are then outlined. We end with a few remarks and several natural questions concerning groupoids, deformations and applications.

\bigskip

\section{Lie groupoids and their operators algebras}

We refer to \cite{Mack,MM} for the classical definitions and
constructions related to groupoids and their Lie algebroids. The construction of the $C^*$-algebra of a groupoid is due to Jean Renault \cite{Ren}, one can look at his course \cite{RenC} which is mainly devoted to locally compact groupoids.

\subsection{Lie groupoids}
\subsubsection{Generalities}
A \emph{groupoid} is a small
category in which every morphism is an isomorphism.  Thus a 
 groupoid $G$ is a pair $(G^{(0)},G^{(1)})$ of sets together  with
structural morphisms:
\begin{description} \item[Units and arrows.] The set $G^{(0)}$ denotes the set of \emph{objects} (or \emph{units}) of the groupoid, whereas the set $G^{(1)}$ is the set of \emph{morphisms} (or \emph{arrows}). The \emph{unit map} $u: G^{(0)} \to G^{(1)}$  is the  injective map which assigns to any object of $G$ its identity morphism. 
\item[The source and range maps] $s,r: G^{(1)} \to G^{(0)}$ are (surjective) maps equal to identity in restriction to $G^{(0)}$: $s\circ u=r\circ u =Id$.
\item[The inverse] $\iota: G^{(1)} \to G^{(1)}$ is an involutive map which exchanges the source and range:
$$\mbox{for } \alpha\in G,\ ( \alpha^{-1})^{-1}= \alpha \mbox{ and } s(  \alpha^{-1})=r( \alpha), \mbox{ where }  \alpha^{-1} \mbox{ denotes } \ \iota( \alpha)$$
\item[The  partial multiplication] $m: G^{(2)} \rightarrow G^{(1)}$ is defined on the set of \emph{composable pairs} $G^{(2)}=\{( \alpha, \beta) \in G^{(1)} \times G^{(1)}\ \vert \ s( \alpha) = r(\beta)\}$. It satisfies for any $( \alpha,\beta)\in G^{(2)}$: 
$$r( \alpha \beta)=r( \alpha),\ s( \alpha \beta)=s(\beta), \   \alpha u(s( \alpha))=u(r( \alpha)) \alpha= \alpha, \ \alpha^{-1} \alpha=u(s( \alpha))  $$
where $\alpha\beta$   stands for  $m(\alpha,\beta)$. Moreover the product is associative, if $\alpha,\beta,\gamma \in G$:
$$(\alpha \beta)\gamma=\alpha(\beta \gamma) \mbox{ when }\ s(\alpha)=r(\beta) \mbox{ and } s(\beta)=r(\gamma)$$

\end{description}
 
We often identify $G^{(0)}$ with its image in $G^{(1)}$ and  make the confusion between $G$ and $G^{(1)}$. A groupoid $G=(G^{(0)},G^{(1)},s,r,u,\iota,m)$ will be simply denoted $G\overset{r,s}{\rightrightarrows} G^{(0)}$ or just $G\rightrightarrows  G^{(0)}$.

\bigskip {\bf Notation.} 
For any maps $f:A\to G^{(0)}$ and $g:B\to G^{(0)}$, define $$G^f=\{(x,\gamma)\in A\times G;\ r(\gamma)=f(x)\} , \ G_g=\{(\gamma,x)\in G \times B;\ s(\gamma)=g(x)\}$$ and $$G^f_g=\{(x,\gamma,y)\in A\times G\times B;\ r(\gamma)=f(x),\ s(\gamma)=g(y)\}\ .$$  In particular for $A,B\subset G^{(0)}$, we put $G^A=\{\gamma\in G;\ r(\gamma)\in A\}$ and $G_A=\{\gamma\in G;\ s(\gamma)\in A\}$; we also put $G_A^B=G_A\cap G^B$.

\begin{remark} For any $x\in G^{(0)}$, $G_x^x$ is a group with unit $x$, called the \emph{isotropy group} at $x$. It acts by left (\resp right) multiplication on $G^x$ (resp. $G_x$) and the quotient identifies with $s(G^x)=r(G_x)\subset G^{(0)}$ which is called the \emph{orbit} of $G$ passing through $x$.
Thus a groupoid acts on its set of units. 

Note that $A$ is a \emph{saturated} subset of $G^{(0)}$ (for the action of $G$) if and only if $G_A=G^A=G_A^A$.
\end{remark}

In order to construct the $C^*$-algebra of a groupoid, we will assume that it is \emph{locally compact.}  This means that $G^{(0)}$  and $G$ are endowed with topologies for which \begin{itemize}
\item $G^{(0)}$ is a locally compact  Hausdorff  space, 
\item $G$ is second countable, and locally compact locally Hausdorff \ie each point $\gamma$ in $G$ has a compact (Hausdorff) neighborhood; 
\item all structural maps ($s,r,u,\iota, m$) are continuous and $s$ is open. 
\end{itemize}
In this situation the map $r$ is open and the $s$-fibers of $G$ are Hausdorff.

\smallskip \noindent In order to study differential operators and index theories, we will assume our groupoid to be smooth. The groupoid $G\rra G^{(0)}$ is \emph{Lie} or \emph{smooth} when
$G$ and $G^{(0)}$ are second countable smooth manifolds with $G^{(0)}$ Hausdorff, $s$ is a smooth
submersion (hence $G^{(2)}$ is a  manifold) and the
structural  morphisms are smooth.

The Lie groupoid $G$ is said to be {\it $s$-connected} when for any $x\in G^{(0)}$, the $s$-fiber of $G$ over $x$ is connected. The $s$-connected component of a groupoid $G$ is $\cup_{x\in G^{(0)}} CG_x$  where $CG_x$ is the connected component of the $s$-fiber $G_x$ which contains the unit $u(x)$. The groupoid $CG$ is the smallest open subgroupoid of $G$ containing its units.

\medskip \begin{examples}\label{examples2.2}
\begin{enumerate} \item A space $M$ is a groupoid over itself with $s=r=u=\mbox{Id}$. Thus, a manifold is a Lie groupoid.

\item A group $H\rightrightarrows \{e\}$ is a groupoid over its unit
  $e$, with the usual product and inverse map. A Lie group is a Lie groupoid!

\item A group bundle: $\pi: E\rightarrow M$ is a groupoid
  $E\rightrightarrows M$ with $r=s=\pi$ and algebraic operations
  given by the  group structure of each fiber $E_x$, $x\in M$. In particular, a smooth vector bundle over a manifold gives thus rise to a Lie groupoid.

\item If $\cR$ is an equivalence relation on a space $M$, then the
  graph of $\cR$, $G_{\cR}:=\{(x,y)\in M\times M \ \vert \ x\cR y\}$, 
  admits a structure of groupoid over $M$,  which is given by:
$$ u(x)=(x,x)\ , \ s(x,y)=y\ ,\ r(x,y)=x \ , \\
  (x,y)^{-1}=(y,x) \ ,\ (x,y)\cdot(y,z)=(x,z) \ $$ for
$x,\ y,\ z$ in $M$. Notice that the orbits of the groupoid $G_{\cR}$ are precisely the orbits of the equivalence relation $\cR$.
If $M$ is a manifold, $G_{\cR}$ is a smooth submanifold of $M\times M$ \emph{and} $s$ restricts to a submersion, it is a Lie groupoid.

\smallskip \noindent When $x\cR y$ for any $x,\ y$ in $M$,
$G_{\cR}=M\times M \rightrightarrows M$ is called the {\it pair
  groupoid}. If $M$ is a manifold, the pair groupoid  $M\times M$ is a Lie groupoid. 
  
\smallskip Without entering too much in the details, let us say that a smooth regular foliation on a manifold $M$(of dimension $n$)  is an equivalence relation on $M$ whose orbits, called the leaves, are immersed connected submanifolds of $M$ (of dimension $p$). The corresponding groupoid $G_{\cR}$ does not have a smooth structure, but there is a \enquote{smallest} Lie groupoid of dimension $n+p$, called the \emph{holonomy groupoid of the foliation}, whose orbits are the leaves \cite{Win,Pra84,Hae,MM}. The holonomy appears to be exactly the obstruction for $G_{\cR}$ to be smooth!

\item If $H$ is a group acting on a space $M$, the {\it groupoid of
  the action} is $H\ltimes M \rightrightarrows M$ with the following
structural morphisms 
$$\begin{array}{cc} u(x)=(e,x)\ , \ s(g,x)=x\ ,\ r(g,x)=g\cdot x \ , \\
  (g,x)^{-1}=(g^{-1},g\cdot x) \ ,\ (h,g\cdot x)\cdot(g,x)=(hg,x) \
  ,\end{array}$$ for 
$x$ in $M$ and $g,\ h$ in $H$. Once again, the notion of isotropy groups, and orbits of the groupoid coincide with the one of the action. 

If $H$ is a Lie group, $M$ is a smooth manifold and the action is smooth, then $H\ltimes M$ is a Lie groupoid.

\item \label{examples2.2.f}Let $M$ be a smooth manifold of dimension $n$. The \emph{Poincaré groupoid} of 
$M$ is $$\Pi(M):=\{\bar{\gamma} \ \vert \ \gamma:[0,1] \rightarrow M \mbox{ a
  continuous path} \} \rightrightarrows M $$ where $\bar{\gamma}$ denotes
the homotopy class (with fixed endpoints) of $\gamma$.  For $x\in M$, $u(x)$ will be the (class of the ) constant path at $x$, $s(\bar{\gamma})=\gamma(0)$, $r(\bar{\gamma})=\gamma(1)$, the product comes from the concatenation product of paths....

The groupoid $\Pi(M)$ is naturally endowed with a
smooth structure (of dimension $2n$). For any $x\in M$, the isotropy group $\Pi(M)_x^x$ is the fondamental group of $M$ with base point $x$ and $\Pi(M)_x$ the corresponding universal covering.

\item If $G\rightrightarrows M$ is a groupoid and $f:N\rightarrow M$ a map, $G_f^f \rightrightarrows N$ is again a groupoid: 
$$ u(x)=(x,f(x),x),\ s(x,\alpha,y)=y,\ (x,\alpha,y)^{-1}=(y,\alpha^{-1},x),\ (x,\alpha,y)(y,\beta,z)=(x,\alpha\beta,z)$$
where $\alpha$, $\beta$  are in $G$, $x,\ y,\ z$ in $N$ and  $f(x)=r(\alpha)$, $f(y)=s(\alpha)=r(\beta)$ and $f(z)=s(\beta)$.

\smallskip When $G$ is a Lie groupoid, and $f$ a smooth map transverse to $G$ (see Definition \ref{transverse}), $G_f^f\rightrightarrows N$ is a Lie groupoid.

\end{enumerate}
\end{examples}

\medskip The infinitesimal object associated to a Lie groupoid is its \emph{Lie
algebroid}:

\begin{definition}\label{algebroid}
A  \emph{Lie algebroid} $\cA$ over a manifold $M$ is a vector bundle
$\cA \to M$, together with a Lie algebra structure on the space
$\Gamma(\cA)$ of smooth sections of $\cA$ and  a bundle map $\varrho:
\cA \rightarrow TM$, called the \emph{anchor}, whose extension to sections of these bundles
satisfies

(i) $\varrho([X,Y])=[\varrho(X),\varrho(Y)]$, and

(ii) $[X, fY] = f[X,Y] + (\varrho(X) f)Y$,

\noindent for any smooth sections $X$ and $Y$ of $\cA$ and any
smooth function $f$ on $M$.
\end{definition}

\smallskip 

Now, let $G  \overset{s,r}{\rightrightarrows}G^{(0)}$ be a Lie
groupoid. 

\smallskip  \noindent
For any $\alpha$ in $G$, let $R_{\alpha}: G_{r(\alpha)} \rightarrow
G_{s(\alpha)}$ be the right multiplication by $\alpha$. A tangent vector
field $Z$ on $G$ is \emph{right invariant} if it  satisfies,

\begin{itemize}
\item[--] $Z$ is $s$-vertical, namely $Ts(Z)=0$, \ie for every $\alpha \in G$, $Z(\alpha)$ is tangent to the fiber $G_{s(\alpha)}$.
\item[--] For all $(\alpha,\beta)$ in $G^{(2)}$, $Z(\alpha  \beta)=TR_{\beta}(Z(\alpha))$.
\end{itemize}

\medskip \noindent The Lie algebroid ${\gA}  G$ of the Lie groupoid $G$ is defined as follows (see \cite{Mack}).
\begin{itemize}
\item[--] The fibre bundle ${\gA}  G \rightarrow G^{(0)}$ is the restriction of the kernel of the differential $Ts$ of $s$
to $\cG^{(0)}$. In other words, $\gA G=\cup_{x\in G^{(0)}} T_x
G_x$ is the union of the tangent spaces to the $s$-fibres at the corresponding unit.
\item[--] The anchor $\varrho:{\gA}  G \rightarrow T\cG^{(0)}$
is the restriction of the differential $Tr$ of $r$
to ${\gA}  G$.
\item[--] If $Y:U \rightarrow {\gA}  G$ is a local section of ${\gA}  G$,
  where $U$ is an open subset of $G^{(0)}$, we define the local \emph{
    right invariant vector field} $Z_Y$ \emph{associated} with $Y$ by
  $$Z_Y(\alpha)=TR_{\alpha}(Y(r(\alpha))) \ \makebox{ for all } \
  \alpha \in G^U \ .$$

\noindent The Lie bracket is then defined by
$$\begin{array}{cccc} [\ ,\ ]: & \Gamma({\gA}  G)\times \Gamma({\gA} 
  G) & \longrightarrow & \Gamma({\gA}  G) \\
 & (Y_1,Y_2) & \mapsto & [Z_{Y_1},Z_{Y_2}]_{G^{(0)}}
\end{array}$$
where $[Z_{Y_1},Z_{Y_2}]_{G^{(0)}}$  is the
restriction of the usual bracket $[Z_{Y_1},Z_{Y_2}]$ to ${G^{(0)}}$.
\end{itemize}

Notice that $\gA G$ identifies with the normal bundle $\cN_{G^{(0)}}^G$ of the inclusion $u:G^{(0)}\hookrightarrow G$.

\begin{definition}\label{transverse} A smooth map $f:N\to M$ is \emph{transverse} to a Lie groupoid $G\rightrightarrows M$ when for all $x\in N$: 
$Tf(T_xN)+\varrho(\gA G)_{f(x)}=T_{f(x)}M$.
\end{definition}

\medskip \noindent Lie theory for groupoids is much trickier than for groups. In particular, a Lie algebroid does not always integrate into a Lie groupoid (see \cite{AlMo} for a counterexample). 

Nevertheless, when the anchor of a Lie algebroid $\cA$ is injective in restriction to a dense open subset it is integrable and there is a \enquote{smallest} $s$-connected Lie groupoid integrating it (\cf \cite{Debord}). This situation is often encountered in index theory where such a Lie algebroid of vector fields is naturally associated with the studied geometrical object (\eg manifolds with corners, conical singularities...). See \cite{CraFer} for a complete answer to this integrability problem.

\medskip \noindent\begin{examples}\label{examples2.5}
\begin{enumerate} \item The Lie algebroid of a Lie group is the Lie algebra of the group.
\item The Lie algebroid of the pair groupoid $M\times M \rightrightarrows M$ on a smooth manifold $M$, is $TM$ with identity as anchor.

\item If $f:M\rightarrow B$ is a smooth submersion, the Lie groupoid of the equivalence relation on $M$ \enquote{being on the same fiber of $f$}  is $B_f^f=M\times_f M \rightrightarrows M$ and its Lie algebroid is the kernel of $Tf$ with anchor the inclusion.\label{examples2.5.b}

\item More generally, if $\cF$ is a regular foliation on a manifold $M$, $T\cF$ with inclusion as anchor, defines a Lie algebroid over $M$ and the \emph{holonomy groupoid} is the smallest Lie groupoid which integrates $T\cF$ \cite{Win,Pra84}.\label{examples2.5.c}

\item Let $M$ be a manifold and $V$ a hypersurface cutting $M$ into two pieces. The module of smooth vector fields on $M$ that are tangent to $V$ was considered by Melrose for the study of \emph{$b$-operators} for manifold with boundary \cite{Mel2}. This module is the module of sections of a Lie algebroid $\gA_b$ over $M$ which integrates into the $b$-groupoid $G_b$ \cite{Month2}.\label{examples2.5.d}

If $M=V\times \R$ and $V=V\times \{0\}$  (which is always locally the case around $V$ up to a diffeomorphism), then $\gA_b=TV\times T\R$ with anchor $\varrho(x,U,t,\xi)\mapsto (x,U,t, t\xi)$ and $G_b=(V\times V) \times (\R\rtimes \R_+^*)$ is the product of the pair groupoid on $V$ with the groupoid of the multiplicative action of $\R_+^*$ on $\R$.
\end{enumerate}
\end{examples}

\begin{remark} The unit spaces of many interesting groupoids have boundaries or corners. In (almost) all the situations, these groupoids sit naturally inside Lie groupoids \emph{without boundaries} as restrictions to closed \emph{saturated} subsets. This means that the object under study is a subgroupoid $G_V^V=G_V$ of a Lie groupoid $G{\rightrightarrows} G^{(0)}$ where $V$ is a closed saturated subset of $G^{(0)}$. Such groupoids, have a natural algebroid, adiabatic deformation, pseudodifferential calculus, \emph{etc.} that are restrictions to $V$ and $G_V$ of the corresponding objects on $G^{(0)}$ and $G$. We chose to give definitions and constructions for Lie groupoids; the case of a longitudinally smooth groupoid over a manifold with corners is a straightforward generalization using a convenient restriction.
\end{remark}

\subsubsection{Morita equivalence of Lie groupoids}

An important feature of noncommutativity is a nontrivial notion of Morita equivalence. Indeed a Morita equivalence of commutative algebras is just an isomorphism.

In the same way, in the world of groupoids there is an interesting notion of Morita equivalence, although this notion reduces to isomorphism both for spaces and for groups.

\begin{definition}

Two Lie groupoids $G_1\overset{r,s}{\rightrightarrows} M_1$ and $G_2\overset{r,s}{\rightrightarrows} M_2$ are \emph{Morita equivalent} if there exists a groupoid $G\overset{r,s}{\rightrightarrows} M$ and smooth maps $f_i:M_i\to M$ transverse  to $G$ such that $f_i(M_i)$ meets all the orbits of $G$ and such that the pull back groupoids $G_{f_i}^{f_i}$ identify to $G_i$.
\end{definition}

More precisely, a Morita equivalence is given by a linking manifold $X$ with extra data: surjective smooth submersions $r:X\to M_1$ and $s:X\to M_2$ and compositions $G_1\times_{s,r}X\to X$,  $X\times_{s,r} G_2\to X$, $X\times_{r,r} X\to G_2$ and $X\times_{s,s} X\to G_1$ with natural associativity conditions (see \cite{MRW} for details). In the above situation, $X$ is the manifold $G_{f_2}^{f_1}$ and the extra data are the range and source maps and the composition rules of the groupoid $G_{f_1\sqcup f_2}^{f_1\sqcup f_2}\rightrightarrows M_1\sqcup M_2$ (see \cite{MRW}).

\bigskip\begin{examples}There are many interesting Morita equivalences of Lie groupoids. 

\begin{enumerate}
\item Given a surjective submersion $f:M\to B$, the subgroupoid $G=\{(x,y)\in M\times M;\ f(x)=f(y)\}$ of the pair groupoid $M\times M$ is Morita equivalent to the space $B$, \ie the groupoid $B\rra B$ -- \cf example \ref{examples2.5}.\ref{examples2.5.b}).

\item More generally, if $G\rra B$ is a Lie groupoid and $f:M\to B$ is a smooth map transverse to $G$ whose image meets all the $G$ orbits, then the groupoid $G_f^f$ is a Lie groupoid Morita equivalent to $G$ -- \cf example \ref{examples2.5}.\ref{examples2.5.c}).

\item If $M$ is a connected manifold, its Poincar\'e groupoid $\Pi(M)$ (\cf example \ref{examples2.2}.\ref{examples2.2.f}) is Morita equivalent to the fundamental group $\pi_1(M)$. 
\end{enumerate}
\end{examples}

\subsection{$C^*$-algebra of a Lie groupoid}

\subsubsection{Convolution $*$-algebra of smooth functions with compact support}

Recall that on a Lie group $H$, the \emph{convolution product formula} is given for $f$ and $g$ in $C^{\infty}_c(H)$ (\ie smooth functions with compact supports on $H$) by $$f\ast g (x)=\int_H f(y)g(y^{-1}x) dy$$
For $M$ a manifold, the \emph{convolution product formula of kernels} is given for $f$ and $g$ in $C^{\infty}_c(M\times M)$ by $$f\ast g (x,y)=\int_M f(x,z)g(z,y) dz$$
Both these convolution products are special cases of convolution product on a Lie groupoid, the first one is for the Lie group viewed as a Lie groupoid $H\rightrightarrows \{e\}$ and the second one corresponds to the pair groupoid $M\times M\rightrightarrows M$.

\medskip Let us assume now that $G\rightrightarrows G^{(0)}$ is a Lie groupoid with source $s$ and range $r$.
We define a \emph{convolution algebra} structure on $C^\infty_c(G)$ in the following way:
\begin{description}
\item[Convolution product.] For $f,\ g\in C^\infty_c(G)$ and $\gamma \in G$ one wants to define a convolution formula of the following form
$$f\ast g(\gamma)=\int_{(\alpha,\beta)\in G^{(2)};\ \alpha \beta=\gamma}f(\alpha)g(\beta)$$
In order to define the previous integral, one can choose a smooth \emph{Haar system} on $G$ that is 
\begin{itemize} \item  a smooth family of Lebesgue measure $\nu^x$ on every $G^x$, $x\in G^{(0)}$, 
\item with left invariance property: for every $\alpha \in G$, the diffeomorphism $\beta \mapsto \alpha \cdot \beta $ from $G^{s(\alpha)}$ with $G^{r(\alpha)}$ sends the measure $\nu^{s(\alpha)}$ to the measure $ \nu^{r(\alpha)}$.
\end{itemize}
Now the convolution formula becomes: 
$$f\ast g(\gamma)=\int_{G^{r(\gamma)}}f(\alpha)g(\alpha^{-1}\gamma)\, d\nu^{r(\gamma)}(\alpha). $$

The convolution product is associative by invariance of the Haar system (and Fubini).
\item[Adjoint.] For $f\in C_c^\infty(G)$, its adjoint is the function $f^*:\alpha \mapsto \overline{f(\alpha^{-1})}$.
\end{description}

\begin{remarks} Let us mention two very useful constructions that appeared in  \cite{ConnesSurvey}. \begin{enumerate}
\item  In order to have intrinsic formulas for the convolution and adjoint, it is suitable to replace the space $C_c^\infty(G)$ by the space of sections of a bundle of (half) densities on the groupoid, more precisely sections of the vector bundle $\Omega^{1/2}=|\Lambda|^{1/2} (\ker Ts\times \ker Tr)$ of half densities of the bundle $\ker Ts\times \ker Tr$. Note that the Haar system is just an invariant section of the bundle $|\Lambda|^{1} (\ker Tr)$ of $1$-densities of the bundle $\ker (\ker Tr)$ and can be used to trivialize the bundle $|\Lambda|^{1/2} (\ker Ts\times \ker Tr)$. 
\item Also in \cite{ConnesSurvey} Connes explains how to naturally define the convolution algebra ``$C_c^\infty(G)$'' when the groupoid $G$ is not assumed to be Hausdorff (but is still a \emph{locally} Hausdorff manifold).
\end{enumerate}

\end{remarks}

\subsubsection{Norm and $C^*$-algebra}

For $f\in C_c^{\infty}(G)$, define 
$$ \|f\|_{I}=\max\big( \underset{x\in G^{(0)}}{\mbox{sup}} \int_{G^x} \vert f(\gamma) \vert d\nu^{r(\gamma)}(\alpha) \ ; \  \underset{x\in G^{(0)}}{\mbox{sup}} \int_{G^x} \vert f(\gamma^{-1}) \vert d\nu^{r(\gamma)}(\alpha) \big)$$

A $*$-representation $\pi : C_c^{\infty}(G) \longrightarrow \cB(\cH)$, where $\cB(\cH)$ is the algebra of bounded operator on the separable Hilbert space $\cH$, is \emph{bounded } when it satisfies:  $\| \pi(f) \| \leq \| f \|_I$ for any $f\in C_c^{\infty}(G)$.

We define the \emph{maximal norm} of $f\in C_c^{\infty}(G)$ by:
$$\|f\|_{max} = \underset{\pi \ bounded}{\mbox{sup}} \| \pi(f) \|_{\cB(\cH)}$$

For any $x\in G^{(0)}$ the map $\pi^x : C_c^{\infty}(G) \rightarrow \cB(L^2(G_x))$ defined by the formula: 
$$\pi^x(f)(\tau)(\gamma)=\int_{G^{r(\gamma)}} f(\alpha)\tau(\alpha^{-1}\gamma)\, d\nu^{r(\gamma)}(\alpha). $$
where  $f\in C_c^{\infty}(G)$, $\tau \in L^2(G_x)$ and $\gamma \in G$, is a bounded representation.

\medskip We define the \emph{minimal norm} of $f\in C_c^{\infty}(G)$ by:
$$\|f\|_{min} = \underset{x\in G^{(0)}}{\mbox{sup}} \| \pi^x(f) \|_{\cB(L^2(G_x))}$$

The \emph{reduced $C^*$-algebra} of $G$ is the completion of $C_c^{\infty}(G)$ with respect to the minimal norm: $C_r^*(G)=\overline{C_c^{\infty}(G)}^{min}$. The \emph{maximal $C^*$-algebra} of $G$ is the completion of $C_c^{\infty}(G)$ with respect to the maximal norm: $C^*(G)=\overline{C_c^{\infty}(G)}^{max}$.

\smallskip The identity induces a surjective morphism from $C^*(G)$ to $C_r^*(G)$.  This morphism is an isomorphism when the groupoid $G$ is \emph{amenable} (see \cite{AnaRen} for a discussion of the amenability of locally compact groupoids).

\bigskip The $C^*$-completions $C^*(G)$ and $C^*_r(G)$ have both advantages and disadvantages. Some properties hold for one of them and not necessarily for the other one. The celebrated Baum-Connes conjecture \cite{BaCo, BaCoHi} is a statement for the $K$-theory of the reduced one -- and Kaszdan's property $T$ shows easily that it cannot hold for the the maximal one (see \cite{Valette} for a very nice discussion on the Baum-Connes conjecture). 

On the other hand, let $G\rra M$ be a Lie groupoid and $X\subset M$ a closed subset of $M$  \emph{saturated} for $G$ (\ie if for $\alpha \in G$, $s(\alpha)\in X$ if and only if $r(\alpha)\in X$). Then we have an exact sequence: $$0\to C^*(G_{M\setminus X})\to C^*(G) \to C^*(G_X)\to 0$$
of maximal $C^*$-algebras. The corresponding sequence at the level of reduced $C^*$-algebras is \emph{not} always exact in its middle term. This nonexactness is responsible for counterexamples to the Baum-Connes conjecture in \cite{HLS}. See \cite{BGW, BEW1, BEW2} for a possible solution to this lack of exactness.

\subsection{Deformation to the normal cone and blowup groupoids}

\subsubsection{Deformation to the normal cone groupoid}

Let us first recall the standard \emph{deformation to the normal cone} construction.  Let $V\subset M$ be a submanifold of a manifold $M$ and denote  $\cN_V^M$ the normal bundle. The \emph{deformation to the normal cone} of $V$ in $M$ is:
$$DNC(M,V)=(M\times \R^*)\sqcup (\cN_V^M\times \{0\}).$$

It is equipped with the natural smooth structure generated by the following constraints:
\begin{itemize}
\item the map $\varphi: DNC(M,V)\to M\times \R$ given by $(x,t)\in M\times \R^* \mapsto (x,t)$ and $(x,\xi,0)\in \cN_V^M\times \{0\} \mapsto (x,0)$ is smooth.
\item if $f:M\to \R$ is any smooth function that vanishes on $V$, the function $f^{dnc}:DNC(M,V)\to  \R$ given by $(x,t)\in M\times \R^*\mapsto \frac{f(x)}{t}$ and $(x,\xi,0)\in \cN_V^M\times \{0\} \mapsto df(\xi)$ is smooth.
\end{itemize}

One can also define the smooth structure with the choice of an exponential map  
\(\theta : U'\subset  \cN_V^M \rightarrow U\subset M\),
by requiring the map $$\Theta : (x,\xi,t) \mapsto \left\{ \begin{array}{l} (\theta(x,t\xi),t) \ for \ t\not=0 \\ (x,\xi,0)\ for \ t=0 \end{array} \right. $$
to be a diffeomorphism from the open neighbourhood \(W'=\{(x,\xi,t)\in \cN_V^M\times \R \ \vert \ (x,t\xi)\in U'\}\)  of  \(\cN_V^M\times \{0\}\) in \(\cN_V^M\times \R\) on its image. 

\medskip Note that describing the smooth structure thanks to the choice of an exponential map ensures that such a structure exists, while the description thanks to the characterisation of the smooth functions ensure that this structure is independent of choices. 

\begin{remark}
By the characterisation of the smooth functions, it follows that the deformation to the normal cone construction is functorial.
\end{remark}

A consequence of this remark is 
\begin{corollary}
 If $G$ is a Lie groupoid and $H$ is a subgroupoid, $DNC(G,H) \rightrightarrows DNC(G^{(0)}, H^{(0)})$ is a Lie groupoid. Moreover the Lie algebroid of $DNC(G,H)$ is $DNC(\gA G, \gA H)$.
\end{corollary}

A the set level $DNC(G,H)=G\times \R^* \sqcup \cN_H^G\times\{0\}$ and the normal space $\cN_H^G$ is equipped with a Lie groupoid structure with units $\cN_{H^{(0)}}^{G^{(0)}}$ called the \emph{normal groupoid} of the inclusion of $H$ in $G$ its Lie algebroid is $\cN_{\gA H}^{\gA G}$.

\medskip \noindent {\bf Example.}  \label{sectionTangGro}
The first and most famous example of groupoid resulting from the $DNC$ construction is the \emph{tangent groupoid} of Connes \cite{ConnesNCG}: let $M$ be a smooth manifold diagonally embedded in the pair groupoid $M\times M$, perform $DNC(M\times M,M)$ and restrict it to $M\times [0,1]$:  
$$G_T(M)=DNC(M\times M,M)\vert_{M\times [0,1]}=M\times M\times (0,1] \sqcup TM\times \{0\} \rightrightarrows M\times [0,1]$$
Similarly, the \emph{adiabatic groupoid} of a Lie groupoid $G\rightrightarrows G^{(0)}$ is the restriction over $G^{(0)}\times [0,1]$ of $DNC(G,G^{(0)})$ \cite{MonthPie,NWX}: 
$$G_{ad}=G\times (0,1] \sqcup \gA G\times \{0\} \rightrightarrows G^{(0)}\times [0,1]$$

\subsubsection{Blowup groupoid}
\label{sec:blowup}

We keep the notation of the previous section: $V\subset M$ is a submanifold of a manifold $M$ and $DNC(M,V)$ is the deformation to the normal cone of $V$ in $M$. 

Recall that $\varphi : DNC(M,V)=M\times \R^*\sqcup \cN_V^M\times \{0\} \rightarrow M\times \R$ is the natural (smooth) map. We will consider the manifold with boundary $DNC_+(M,V)=\varphi^{-1}(M\times \R_+)=M\times \R^*_+\sqcup \cN_V^M$.

\smallskip  The \emph{scaling action} of \(\R_+^*\) on \(M\times \R^*\) extends to the \emph{zooming action} of  \(\R_+^*\) on \(DNC_+(M,V)\) :
\[  \begin{array}{ccc} \DNC_+(M,V)\times \R^* & \longrightarrow & \DNC_+(M,V) \\ (z,t,\lambda) & \mapsto & (z, \lambda t) \ for \ t\not=0  
\\    (x,X,0,\lambda) & \mapsto & (x,  \frac{1}{\lambda} X,0) \ for \ t=0  \end{array} 
\]

By functoriality, the manifold $V\times \R_+=DNC_+(V,V)$ embeds in \(\DNC_+(M,V)\). The zooming action is free and proper on the open subset  \(\DNC_+(M,V)\setminus V\times \R_+ \) of $\DNC_+(M,V)$.  We let the \emph{spherical blowup} of $V$ in $M$ be: 

\[\SBlup(M,V)=\big( \DNC_+(M,V)\setminus V\times \R_+ \big)/\R_+^*=M\setminus V \cup \mathbb{S}(N_V^M)\ .\]

\begin{remark}
The spherical blowup construction is functorial \enquote{wherever it is defined}. Precisely, suppose that we have a commutative diagram
$\xymatrix{V\ar@{^{(}->}[r]\ar[d]_{f\vert_V}&M\ar[d]^{f}\\V'\ar@{^{(}->}[r]&M'}$ where the horizontal arrows are embedings of submanifolds. Functoriality of the deformation to the normal cone constructions yields a smooth map $\DNC(f):DNC(M,V)\to DNC(M',V')$. This smooth map restricts to the map $\DNC(f)_+:DNC_+(M,V)\to DNC_+(M',V')$ which is equivariant under the zooming action.

\smallskip Let $SU_f(M,V)=\DNC_+(M,V)\setminus \DNC(f)^{-1}(V'\times \R_+)$ and define \[\SBlup_f(M,V)=SU_f/\R_+^*  \subset \SBlup(M,V)\]
Then $\DNC(f)_+$ passes to the quotient:
\[\SBlup(f):\SBlup_f(M,V)\to \SBlup(M',V').\]

\end{remark}

If $\xymatrix{V\ar@{^{(}->}[r]\ar[d]_{g\vert_V}&M\ar[d]^{g}\\V'\ar@{^{(}->}[r]&M'}$ is another smooth \emph{map of embedings} we will denote $SU_{f,g}(M,V)=\DNC_+(M,V)\setminus (\DNC(f)^{-1}(V'\times \R_+)\cup \DNC(g)^{-1}(V'\times \R_+))$ and $\SBlup_{f,g}(M,V)$ its quotient under the zooming action.

\smallskip A consequence of the preceding remark is 
\begin{corollary}
 If $G\overset{r,s}{\rightrightarrows} G^{(0)}$ is a Lie groupoid and $H$ is a subgroupoid, $\SBlup_{r,s}(G,H) \rightrightarrows \SBlup(G^{(0)}, H^{(0)})$ is a Lie groupoid. Moreover the Lie algebroid of $SBlup_{r,s}(G,\Gamma)$ is $SBlup_{r,s}(\gA G,\gA \Gamma)$.
\end{corollary}

\medskip \noindent {\bf Examples}
\begin{enumerate} \item
If $G\rightrightarrows G^{(0)}$ is a Lie groupoid define $\G=G\times \R\times \R \rightrightarrows G^{(0)}\times \R$, the product of $G$ with the pair groupoid on $\R$ . 
One can check that 
$$\SBlup_{r,s}(\G,\G^{(0)}\times \{(0,0)\})=\DNC_+(G,G^{(0)})\rtimes \R_+^* \rightrightarrows G^{(0)}\times \R$$ 
and recover the \emph{Gauge adiabatic groupoid} of \cite{DS1}.

\item Let $V\subset M$ be a hypersurface. The blowup procedure enables to recover groupoids and spaces involved in the pseudodifferential calculus on manifold with boundary. In particular: 
$$\underbrace{G_b=\SBlup_{r,s}(M\times M,V\times V)}_{\mbox{The b-calculus groupoid}} \subset \underbrace{\SBlup(M\times M,V\times V)}_{\mbox{Melrose's b-space}}$$ 
$$\underbrace{G_0=\SBlup_{r,s}(M\times M,\Delta(V))}_{\mbox{The 0-calculus groupoid}} \subset \underbrace{\SBlup(M\times M,\Delta(V))}_{\mbox{Mazzeo-Melrose's 0-space}}$$

\item One can iterate these constructions to go to the study of manifolds with corners, or consider a foliation with no holonomy on $V$, or define the holonomy groupoid of a manifold with iterated fibered corners \emph{etc.}
\end{enumerate}

\section{Pseudodifferential calculus on Lie groupoids}

 The ``classical'' pseudodifferential calculus was developed in the 1960's and was crucial in the Atiyah-Singer index theorem (\cite{AtSing1}). 

Pseudodifferential operators appear naturally when trying to solve (elliptic) differential equations. Using Fourier transform, one associates canonically to a differential operator a polynomial function - its symbol. The composition of differential operators is not commutative and therefore it does not induce just the product of these polynomials, but at least the leading term of the symbol of the product is the product of the leading terms of the symbols. When trying to solve such an equation, one then naturally tries to invert this symbol. This inverse is no longer a polynomial of course, but one can still associate to it an operator - a pseudodifferential operator.

\medskip The pseudodifferential operators are used in many different parts of mathematics. Information on pseudodifferential operators and much more can be found in the classical books \cite{Hormander1, Hormander2, Hormander3, Hormander4, Shubin, Taylor, Treves1, Treves2}.

\medskip Here we will concentrate to the use of pseudodifferential operators in connection with Lie groupoids and non commutative geometry.

\medskip In \cite{ConnesLNM}, Alain Connes, in order to generalize the Atiyah-Singer index theorem for families (\cite{AtSing4}) to the case of general foliations, considered the $C^*$-algebra of the holonomy groupoid as a noncommutative generalization of the space of parameters and studied index problems with values in this algebra. He therefore introduced the pseudodifferential calculus on the holonomy groupoid of a foliation. 

This pseudodifferential calculus was easily extended to general Lie groupoids (see \cite{MonthPie, NWX}). In this way one constructs an analytic index map $K_*(C_0(\gA^* G))\to K_*(C^*(G))$ for every Lie groupoid $G$.

\medskip Alain Connes made another beautiful observation. His tangent groupoid that we described in section \ref{sectionTangGro} can be used in order to construct the analytic index of elliptic operators in a differential and pseudo-differential free way. The fact that this indeed coincides with the analytic index  of elliptic operators is just a consequence of the existence of a pseudodifferential calculus on every Lie groupoid.

\medskip In this section we will discuss various constructions of this pseudodifferential calculus on groupoids and the construction of the index.

\subsection{Distributions on $G$ conormal to $G^{(0)}$}

The point of view developed by Connes is the following: locally the foliation looks like a fibration. On a \emph{foliation chart} $\Omega_i\simeq U_i\times T_i$, where $T_i$ is the local transversal and $U_i$ represents the leaf direction, a pseudodifferential operator $P_i$ is a family indexed by $T_i$ of operators on $U_i$ (in the sense of \cite{AtSing4}). Connes then defines a pseudodifferential operator on the foliation as a finite sum $f+\sum_i P_i$ of such local pseudodifferential families $P_i$ (with compact support) and $f\in C_c^\infty(G)$.

It was then quite easy to extend this calculus to a general Lie groupoid and this was done independently in \cite{MonthPie} and \cite{NWX}. There, pseudodifferential operators appear as $G$-invariant pseudodifferential families acting on the source fibers of $G$. 

In fact, it is probably easier and more natural to consider pseudodifferential operators on Lie groupoids as distributions on $G$ which are conormal to $G^{(0)}$. This point of view appears in various calculi of Melrose (on some spaces that contain Lie groupoids as dense open subsets -- see \eg \cite{Melbook}) and in \cite{AndrSk2}. It is explained and developed in \cite{LMV,LescVas1} where the interested reader will have all details  and complete description of the pseudodifferential calculus on a Lie groupoid.

\bigskip
Given a manifold $M$ and a (locally) closed submanifold $V$, conormal distributions are some particular distributions on $M$, with singular support in $V$.

\paragraph{A remark on densities.} As the elements of the convolution algebra of a groupoid are sections of a density bundle $\Omega^{1/2}$ rather than functions, the distributions that we consider are generalised sections of the same density bundle. The distribution associated with a smooth section of a bundle $E$ over a manifold $M$ is a continuous linear mapping on the topological vector space $C_c^{\infty}(M;\Omega^1(M)\otimes E^*)$ of smooth sections with compact support of the tensor product of the bundle of one densities on $M$ with the dual bundle of $E$. In order to simplify our exposition we will drop all these trivial bundles and just consider functions - although this issue is not completely trivial. We will in fact assume that coherent choices of sections of these bundles are made.

A guiding principle is that symbols (of scalar operators) are indeed functions.

\subsubsection{Symbols and conormal distributions} 

\paragraph{Symbols and conormal distributions on $\R^n$.}
Let $\alpha=(\alpha_1,\ldots,\alpha_n)\in \N^n$. Put $|\alpha|=\sum\alpha_i$. The map $D_\alpha:f\mapsto \frac{\partial^{|\alpha|}f}{(\partial x_1)^{\alpha_1}\ldots (\partial x_n)^{\alpha_n}}(0)$ is a distribution on $\R^n$ with (singular) support $0$. In Fourier terms it can be written as $D_\alpha(f)=\frac{1}{(2\pi )^n}\int_{\R^n} (i\xi)^\alpha \hat f(\xi)\,d\xi$.

A \emph{classical symbol} on $\R^n$ of order $m\in \Z$ (or in $\C$) is a function $a$ on $\R^n$ that can be written as $$a(\xi)\sim \sum_{k=0}^{+\infty}a_{m-k}(\xi),$$ where $a_j$ is a smooth function on $\R^n\setminus\{0\}$ homogeneous of degree $j$, \ie such that, for $t\in \R_+^*$ and $\xi \in \R^n\setminus\{0\}$ we have $a_j(t\xi)=t^ja_j(\xi)$.

The notation $\sim $ means that for every $k\in \N$, and every $\alpha\in \N^n$, there is a constant $M_{k,\alpha}$ such that, for $\|\xi\|\ge 1$, we have $$\frac{\partial^{|\alpha|}(a-\sum_{j=0}^{k-1}a_{m-j})}{(\partial \xi_1)^{\alpha_1}\ldots (\partial \xi_n)^{\alpha_n}}(\xi)\le M_{k,\alpha}\|\xi\|^{m-k-|\alpha|}.$$
Such a symbol $a$ gives rise to a distribution using the formula $$P_a(f)=\frac{1}{(2\pi )^n}\int_{\R^n}a(\xi)\hat f(\xi)\,d\xi.$$
\begin{itemize}
\item The function $a$ is called the \emph{total symbol} of $P_a$. It is the Fourier transform of $P_a$ since $a(\xi)=P_a(h_\xi)$ where $h_\xi(x)=e^{i\langle x|\xi\rangle}$. 
\item The homogeneous function $a_m$ is called the \emph{principal symbol} of $P_a$. Note that $a_m(\xi)=\lim_{t\to +\infty}t^{-m}a(t\xi)$ - and therefore $a_m$ only depends on $P_a$.
\item The support of the distribution $P_a$ is now $\R^n$, but its \emph{singular support} is $0$: For every neighborhood $V$ of $0$, one can write $P_a=Q_a+\kappa$ where $Q_a$ has support in $V$ and $\kappa$ is a smooth (Schwartz) function. 

\end{itemize}

\paragraph{Symbols on a vector bundle.}
 Let now $p:E\to B$ be a real vector bundle over a manifold $B$. We consider symbols on $E$ as being families - indexed by $B$ of symbols on the fibers $(E_x)_{x\in B}$. Such a symbol is then a function $a:E^*\to \C$ where $E^*$ is the dual vector bundle such that $$a(x,\xi)\sim \sum_{k=0}^{+\infty}a_{m-k}(x,\xi),$$ where $a_j$ is a smooth function on $E^*\setminus B$ (where $B\subset E^*$ is the $0$ section of the bundle $E^*$) homogeneous of degree $j$ in $\xi$, \ie such that, for $t\in \R_+^*$, $x\in B$ and $\xi \in E_x^*\setminus\{0\}$ we have $a_j(x,t\xi)=t^ja_j(x,\xi)$.

The writing $\sim $ means here that, in local coordinates, putting $B=\R^p$ and $E=\R^p\times \R^n$, for every $(k,\alpha,\beta,K)$ where $k\in \N$, $\alpha\in \N^n,\ \beta\in \N^p$ and $K\subset B$ is a compact subset, 

 there is a constant $M_{k,\alpha,\beta,K}$ such that, for $x\in K$ and  $\|\xi\|\ge 1$, we have $$\frac{\partial^{|\alpha|+|\beta| }(a-\sum_{j=0}^{k-1}a_{m-j})}{(\partial x_1)^{\beta_1}\ldots (\partial x_p)^{\beta_p}\,\partial \xi_1)^{\alpha_1}\ldots (\partial \xi_n)^{\alpha_n}}(x,\xi)\le M_{k,\alpha,\beta,K}\|\xi\|^{m-k-|\alpha|}.$$

\begin{remark}
 Note that giving a symbol $a$ of order $m$ on $E^*$ is equivalent to a homogeneous smooth function $b$ of order $m$ on $E^*\times \R_+\setminus B\times \{0\}$. \begin{itemize}
\item given a homogeneous smooth function $b:E^*\times \R_+\setminus B\times \{0\}\to \R$, put $a(x,\xi)=b(x,\xi,1)$;
\item given a symbol $a$, put $b(x,\xi,t)=t^{-m}a(x,\dfrac{\xi}{t})$ for $t\ne 0$ and $b(x,\xi,0)=a_m(u,\xi)$ (where $a_m$ is the principal symbol of $a$.)
\end{itemize}
The expansion $a(x,\xi)\sim \sum_j a_{m-j}(x,\xi)$ corresponds to the Taylor expansion of $b$ at $t=0$: we have $b(x,\xi,t)\sim \sum_j t^ja_{m-j}(x,\xi)$.
\end{remark}

\paragraph{Associated conormal distributions.}

To such a symbol, we may still associate a distribution $P_a$ given by the formula $$P_a(f)=\frac{1}{(2\pi )^n}\int_{E^*}a(x,\xi)\hat f(x,\xi)\,dx\, d\xi$$ (where $\hat f(x,\xi)=\int_{E_x}e^{-i\langle u|\xi\rangle}f(x,u)\,du$).

\begin{itemize}
\item The \emph{singular support} of this distribution is contained in $B\subset E$.

\item The function $a$ is called the \emph{total symbol} of $P_a$. The homogeneous function $a_m$ is called the \emph{principal symbol} of $P_a$. 
\end{itemize}

\paragraph{Symbols and conormal distributions on a manifold.} Let now $M$ be a manifold and $B\subset M$ a closed submanifold of $M$. The tubular neighborhood construction provides us with a neighborhood $U$ of $B$ in $M$ and a diffeomorphism $\varphi :N\to U$ where $N=\cN_B^M$ is the normal bundle: for $x\in B$, we have $N_x=T_xM/T_xB$. The requirement for such a diffeomorphism is $\varphi (b)=b$ for every $b\in B$ and, for every $x$ in $B$, the differential $df:T_xN\to T_xM$ satisfies $p_x\circ df_x(\xi)=\xi$ for $\xi \in N_x\subset T_x(N)$ where $p_x:T_xM\to N_x=T_xM/T_xB$ is the projection.

Using $\varphi$, we obtain a family of distributions on $M$: those that are a sum $Q=\varphi_*(P_a)+\kappa$ of a smooth function $\kappa$ on $M$ (with compact support) and a conormal distribution $P_a$ on $N$ where $a$ is a symbol on $N^*$: we write $$Q(f)=\int_M \kappa (y)f(y)\,dy+\int _{N^*}a(x,\xi)\Big(\int_{u\in N_x}e^{-i\langle u|\xi\rangle}f\circ \varphi(x,u)\,du\Big)\,dx\,d\xi.$$

\paragraph{Diffeomorphism invariance of conormal distributions and principal symbol.}
It turns out that \begin{itemize}
\item the space of distributions on $M$ of the form $\varphi_*(P_a)+\kappa$ does not depend on the partial diffeomorphism $\varphi:N_B^M\to M$;
\item the principal symbol $a_m$ of $\varphi_*(P_a)+\kappa$ does not depend on $\varphi$ either.
\end{itemize}

\paragraph{Remark.} We can also write $$Q(f)=\lim_{R\to\infty}\int_M \kappa _R(y)f(y)\,dy\ \ \hbox{with}\ \ \kappa_R(y)=\kappa(y)+\chi(y)\int _{\xi\in N_{p(y)}^*,\ \|\xi\|\le R}a(p(y),\xi) e^{-i\langle \theta(y)|\xi\rangle} \,d\xi,$$ where $\chi\in C^\infty(M)$ is a smooth function which is equal to $0$ outside $U$ and to $1$ for $y\in B$ (taking into account the jacobian of $\varphi$), and where we have written $\varphi^{-1}(y)=(p(y),\theta(y))$ with $p(y)\in B$ and $\theta(y)\in (\cN_B^M)_{p(y)}$.

\paragraph{Conormal distributions.} We denote by $\cI_c^m(M,B)$ the space of such conormal distributions with compact support. The map $\sigma_m:\varphi_*(P_a)+\kappa \mapsto a_m$ is the principal symbol map.

\medskip Note that we have a natural exact sequence $0\to \cI_c^{m-1}(M,B)\to \cI_c^m(M,B)\overset{\sigma_m}{\longrightarrow}C^\infty_m((\cN_B^M)^*\setminus B)\to 0$ where we denoted by $C^\infty_m((\cN_B^M)^*\setminus B)$ the space of smooth functions on $(\cN_B^M)^*\setminus B$ which are homogeneous of degree $m$.

\begin{definition}
A (classical) pseudodifferential operator of order $m$ on a Lie groupoid $G\rra M$ is a conormal distribution $P\in \cI_c^m(G,M)$.
\end{definition}

\subsubsection{Convolution}

Let $G\rra M$ be a Lie groupoid.

The convolution for $C^\infty_c(G)$ can be understood in the following way. Let $G^{(2)}=\{(\alpha,\beta)\in G\times G;\ s(\alpha)=r(\beta)\}$ be the set of composable elements, and let $p_1,p_2,m:G^{(2)}\to G$ be the maps defined by $p_1(\alpha,\beta)=\alpha$, $p_2(\alpha,\beta)=\beta$ and $m(\alpha,\beta)=\alpha\beta$.

Take $f_1,f_2\in C_c^\infty(G)$. We then may write $f_1\ast f_2=m_!(p_1^*(f_1) . p_2^*(f_2))$, where:
\begin{itemize}
\item $p_i^*:C^\infty (G)\to C^\infty(G^{(2)})$ is given by $p_i^*(f_i)=f_i\circ p_i$;
\item $p_1^*(f_1) . p_2^*(f_2)$ is just the pointwise product of the functions $p_i^*(f_i)$;
\item $m_!:C_c^\infty (G^{(2)})\to C_c^\infty(G)$ is the integration along the fibers of the submersion $m$: 

$m_!(f)(\gamma)=\int_{\alpha \beta=\gamma}f(\alpha,\beta)\,d\nu=m_!(f)(\gamma)=\int_{G^{r(\gamma)}} f(\alpha,\alpha^{-1}\gamma) d\nu^{r(\gamma)}(\alpha)$.
\end{itemize}
We wish to extend these three operations to the case when $f_1$ and $f_2$ are conormal distributions. 

\subsubsection*{Push-forward, pull-back,  product of distributions}

\begin{dmc}{Push-forward}
Let $\varphi:M\to M'$ be a smooth map. Dual to $\varphi^*:C^\infty(M')\to C^\infty(M)$ is a map $\varphi_*:C_c^{-\infty}(M)\to C_c^{-\infty}(M')$ given by $(\varphi_*(P))(f)=P(\varphi^*(f))$  for $P\in C_c^{-\infty}(M)$ and $f\in C^{\infty}(M')$.

\medskip Let $V\subset M$ be a submanifold. Assume that $\varphi$ is a submersion and that the restriction of $\varphi$ to $V$ is a diffeomorphism. Then, the image $\varphi_*$ of $P\in \cI^m_c(M,V)$, is a smooth distribution:  $\varphi_*(P)\in C_c^{\infty}(M')\subset C_c^{-\infty}(M')$.

We may indeed, by restricting to a neighborhood of $V$ in $M$, assume that $\varphi:M\to V\simeq M'$ is a vector bundle projection and that $P=P_a$ where $a$ is a symbol on the dual bundle $M^*$. Then $P$ is in fact a family of pseudodifferential operators $(P_x)_{x\in V}$ and $$\varphi_*(P)(f)=\int_V f(x)P_x(1)\,dx=\int_V f(x)a(x,0)\,dx.$$ In other words, $\varphi_*(P)$ is the distribution associated with the function $x\mapsto a(x,0)$.
\end{dmc}

\begin{dmc}{Pull-back by a submersion}
Let $p:M'\to M$ be a submersion. Dual to the integration along the fibers $p_!:C_c^\infty(M')\to C_c^\infty(M)$ is a map $p^!:C^{-\infty}(M)\to C^{-\infty}(M')$ given by $(p^!(P))(f)=P(p_!(f))$ for $P\in C^{-\infty}(M)$ and $f\in C_c^{\infty}(M')$. The map $p^!$ extends to distributions the map $p^*:C^{\infty}(M)\to C^{\infty}(M')$.

\begin{proposition} Let $p:M'\to M$ be a submersion. Let  $V\subset M$ be a closed submanifold. Put $V'=p^{-1}(V)$. 
\begin{enumerate}
\item The normal bundle of $V'$ in $M'$ identifies with the pull back of the normal bundle of $V$ in $M$. 
\item Assume that $P\in \cI_c^m(M,V)$. Then $p^!(P)\in \cI_c^m(M',V')$. Under the identification of the normal bundle $N'$ of $V'$ in $M'$ with $p^*\cN_V^M$, the principal symbol of  $p^*P$ is given by  $\sigma_{m}(p^*P)=\sigma_{m}(P)\circ p$.
\end{enumerate}

\begin{proof}
The first statement is obvious. Thanks to it, we may assume that $M$ is a vector bundle $p:E\to V$ and $M'$ is the pull back vector bundle $E'=E\times_VV'\to V'$. The second statement is then immediate too.
\end{proof}
\end{proposition}
\end{dmc}

\begin{dmc}{Products of conormal distributions}
\begin{enumerate}
\item One can extend the (pointwise) product of functions to the case where one of them is a distribution. The pointwise product of two distributions is not always well defined.

Already in this way, we may define the convolution of a classical pseudodifferential operator $P$ with an element $f\in C_c^\infty(G)$: we have $P\ast f=m_*(p_1^!(P) . p_2^*(f_2))$. Then $p_1^!(P) . p_2^*(f_2)\in \cI_c(G^{(2)},p_1^{-1}G^{(0)})$ and the submersion $m:G^{(2)}\to G$ induces a diffeomorphism $p_1^{-1}(G^{(0)})\to G$, whence $P\ast f\in C_c^{\infty}(G)$ and in the same way, $f\ast P\in C_c^{\infty}(G)$.

From this it follows that pseudodifferential operators define \emph{multipliers} of $C_c^\infty(G)$.

\item To explain the convolution of two pseudodifferential operators we use the two following facts which reduce to linear algebra.

\begin{enumerate}
\item Let $M$ be a manifold, $V_1,V_2$ two closed submanifolds of $M$ that are transverse to each other. This means that, for every $x\in V_1\cap V_2$ we have $T_xM=T_xV_1+T_xV_2$ (we do not assume that this sum is a direct sum). Then if $Q_1\in \cI_c^{\ell_1}(M,V_1)$ and $Q_2\in \cI_c^{\ell_2}(M,V_2)$, then the distribution $Q_1.Q_2$ makes sense.
\item If moreover $Q_1.Q_2$ has compact support and $m:M\to M'$ is a submersion whose restriction to both $V_1$ and $V_2$ is a diffeomorphism $V_i\to M'$, then $m_*(Q_1.Q_2)\in \cI_c^{\ell_1+\ell_2}(M',m(V_1\cap V_2))$ and its principal symbol is the product of the symbols of $Q_1$ and $Q_2$ under the natural identification of the normal bundles: \begin{itemize}
\item the restriction $(\cN_{V_1}^M)_{|V_1\cap V_2}$ of $\cN_{V_1}^M$ to $V_1\cap V_2$ identifies with $\cN_{V_1\cap V_2}^{V_2}$;
\item the restriction $(\cN_{V_2}^M)_{|V_1\cap V_2}$ of $\cN_{V_1}^M$ to $V_1\cap V_2$ identifies with $\cN_{V_1\cap V_2}^{V_1}$; 
\item finally, using the map $m$, we identify $\cN_{V_1\cap V_2}^{V_1}$ and $\cN_{V_1\cap V_2}^{V_2}$ identify with $\cN_{m(V_1\cap V_2)}^{M'}$.
\end{itemize}
\end{enumerate}
Given $P_1\in \cI_c^{\ell_1}(G,G^{(0)})$ and $P_2\in \cI_c^{\ell_2}(G,G^{(0)})$ we then put \begin{itemize}
\item $M=G^{(2)}=\{(\alpha,\beta)\in G\times G;\ s(\alpha)=r(\beta)\}$;
 \item $Q_i=p_i^{!}(P_i)$ where $p_i:G^{(2)} \to G$ are the submersions $(\gamma_1,\gamma_2)\mapsto \gamma_i$;
 \item $M'=G$ and $m:(\alpha,\beta)\mapsto \alpha \beta$ is the composition.
\end{itemize}
It follows that $P_1P_2\in \cI_c^{\ell_1+\ell_2}(G,G^{(0)})$ with principal symbol $\sigma_{\ell_1}(P_1)\sigma_{\ell_2}(P_2)$.
\end{enumerate}
\end{dmc}

\paragraph{Formal adjoint.} One can also define the adjoint of a pseudodifferential operator $P$ by setting $P^*=j_*(\overline{P})$, where $j:G\to G$ is the diffeomorphism $\gamma \mapsto \gamma^{-1}$.

\subsubsection{Pseudodifferential operators of order $\le 0$}

Pseudodifferential operators with compact supports on a Lie groupoid $G\rra M$ appear as multipliers of $C_c^\infty(G)$. 

\begin{proposition} \label{negpdo}
Pseudodifferential operators with compact support of order $\le 0$ extend to multipliers of $C^*(G)$; pseudodifferential operators of order $<0$ are in fact elements of $C^*(G)$.
\end{proposition}

The first statement means that if $P$ is a pseudodifferential operator with compact support in $G$ and of order $\le 0$, then there exists a constant $c$ such that, for all $f\in C_c^\infty(G)$, we have $\|P\ast f\|\le c\|f\|$ and $\|f\ast P\|\le c\|f\|$ (this is true for both the maximal and the reduced $C^*$-norm of $G$).

\medskip \begin{proof} To establish this statement, first assume that $P$ is of order $<-p$ where $p=\dim G-\dim M$ is the dimension of the algebroid. Note that if $a$ is a symbol of order $< -p$, then $P_a$ is a continuous function. Therefore $P$ is a continuous function with compact support on $G$, and thus an element of $C^*(G)$.

If $P$ is of order $<-p/2$, then $\|P\ast f\|^2=\|f^*\ast P^*\ast P\ast f\|$ (and $\|f\ast P\|^2=\|f\ast P\ast P^*\ast f^*\|$) and as $P^*\ast P$ is of order $<-p$, it is in $C^*(G)$ and thus $\|P\ast f\|^2\le \|P^*\ast P\|\|f\|^2$. It follows that $P$ is a multiplier, and as $P^*\ast P\in C^*(G)$ we find $P\in C^*(G)$.

If $P$ is of negative order, $(P^*P)^{2^k}\in C^*(G)$ for some $k\in \N$, and by induction in $k$, $P\in C^*(G)$.

\medskip Let $P$ be a pseudodifferential operator of order $0$. 

Note first that every smooth function $q\in C_c^\infty(M)$ is a pseudodifferential operator of order $0$ with principal symbol $\sigma_q:(x,\xi)\mapsto q(x)$  - and of course a bounded multiplier: we have $(q\ast f)(\gamma)=q(r(\gamma))f(\gamma)$ and $(f\ast q)(\gamma)=f(\gamma)q(s(\gamma))$.

Let $q\in C_c(M)$ which is equal to $1$ on the support of $\sigma_P$ - \ie the projection on $M$ of the closure of $\{(x,\xi);\ \sigma_q(x,\xi)\ne 0\}$ (which is assumed to be compact in the space of half lines of the bundle $\gA^*$). Let $c\in \R_+$ with $c>\sigma_P(x,\xi)$ for all $(x,\xi)$. Put $b(x,\xi)=q(x)\sqrt{c^2+1-|\sigma_q(x,\xi)|^2}$, and let $Q$ be a pseudodifferential operator with principal symbol $b$. Then $P^*P+Q^*Q$ which has symbol $(1+c^2)|q|^2$ is of the form $(1+c^2)|q|^2+R$ where $R$ is of negative order and therefore $P^*P+Q^*Q$ is bounded. 

For all $f\in C_c(G)$, $\|Pf\|^2=\|f^*P^*Pf\|\le \|f^*P^*Pf+f^*Q^*Qf\|\le \|P^*P+Q^*Q\|\|f\|^2$, and thus $f\mapsto Pf$ is bounded.

In the same way $f\mapsto fP$ is bounded.
\end{proof}

As a consequence, one gets:

\begin{theorem}\label{thmExacSeq}
We have a short exact sequence of $C^*$-algebras $$0\to C^*(G)\to \Psi^*(G)\overset{\sigma}\to C_0(S\gA^*)\to 0$$ where $\Psi^*(G)$ is the closure of the algebra of order $0$ pseudodifferential operators in the multiplier algebra of $C^*(G)$ and $S\gA^*$ is the sphere bundle (the set of half lines) of the dual $\gA^*$ of the algebroid $\gA$ of $G$.
\end{theorem}

Note that this statement is true for the full groupoid $C^*$-algebra as well as for the reduced one.

\begin{proof}
First note that as $\Psi^*(G)$ contains $C_c^\infty(G)$, it contains its closure $C^*(G)$. 

The only statement which does not follow from prop. \ref{negpdo} is that the principal symbol map is well defined, \ie if $\sigma(P)\ne 0$, then $P\not\in C^*(G)$ (both for the full and reduced norm -- it is enough to check this for the reduced one).

In fact, one shows that for every pseudo differential operator $P$ of order $0$ every $x\in M$ and a nonzero $\xi \in \gA^*_x$, we have $\sigma_P(x,\xi)=\lim _{n\to \infty}\langle \varphi_n,\lambda_x(P)\varphi_n\rangle$ where $\varphi_n$ is a function on $G_x$ of $L^2$-norm $1$ whose support is concentrated around $x$ and whose Fourier transform is concentrated in $\R_+^*\xi$:  we may take, in local coordinates, $\varphi_n(y)=(2n)^{p/4} e^{-n\pi \|x-y\|^2-in\langle (y-x)|\xi\rangle}$. Here, $\lambda_x$ is the representation of $C^*(G)$ on $L^2(G_x)$ by left convolution - extended to the multipliers.

On the other hand, for $f\in C_c^\infty(G)$, we have $\lim _{n\to \infty}\langle \varphi_n,\lambda_x(f)\varphi_n\rangle=0$ and by continuity the same is true for $f\in C^*_r(G)$.
\end{proof}

\subsubsection{Analytic index}

The connecting map of the exact sequence of theorem \ref{thmExacSeq} is the \emph{analytic index} of the Lie groupoid $$\partial_G:K_{i+1}(C_0(S\gA^*))\to K_i(C^*(G)).$$

This analytic index can be improved a little by taking vector bundles into account.  Indeed, the starting point of an index problem is often a pair of bundles $E^{\pm}$ over $M$ together with a symbol of order $0$ which gives a smooth family of isomorphisms $a(x,\xi):E^+_x\to E^-_x$. 

Such a symbol defines an element in the \emph{relative $K$-theory}  of the morphism $\mu:C_0(M)\to C_0(S\gA^*)$, in other words an element of the $K$-theory of the \emph{mapping cone} $$C_\mu=\{(f,g)\in C_0(M)\times C_0(S\gA^*\times \R_+); \ \forall (x,\xi)\in S\gA^*,\ g(x,\xi,0)=f(x)\}$$
of $\mu$. This mapping cone is naturally isomorphic to $C_0(\gA^*)$ using the map $(x,\xi,t)\mapsto (x,t\xi)$.

Consider the morphism $\psi:C_0(M)\to \Psi^*(G)$ which associates to a (smooth) function $f$ the order $0$ (pseudo)differential operator multiplication by $f$. Note that we have $\mu=\sigma\circ \psi$. Using the commutative diagram $$\xymatrix{C_0(M)\ar[r]^\mu\ar[d]^\psi &C_0(S\gA^*)\ar@{=}[d]\\\Psi^*(G)\ar[r]^\sigma &C_0(S\gA^*)}$$ we obtain a morphism $\tilde\psi:C_0(\gA^*)=C_\mu\to C_\psi$. Now, we also have an exact sequence $$0\to C^*(G)\overset{e_G}\lra C_\psi\lra C_0(S\gA^*\times \R_+)\to 0.$$ As the algebra $C_0(S\gA^*\times \R_+)$ is contractible (and nuclear), the \emph{excision map} $e_G$ is a $KK$-equivalence. 

\begin{definition}
The \emph{analytic index map} of the Lie groupoid $G$ is the composition $$\ind_G=[e_G]^{-1}\circ [\tilde \psi]:K^j(\gA^*)=K_j(C_0(\gA^*))\to K_j(C^*(G)).$$

\end{definition}

The index $\partial _G$ is the composition of the morphism $K_{j+1}(C_0(S\gA^*))\to K_j(C_0(\gA^*))$ induced by the inclusion $S\gA^*\times \R_+^*\to \gA^*$ with the index map $\ind_{G})$.

\subsection{Classical examples}

The analytic index for groupoids recovers many classical situations.

\begin{enumerate}
\item Assume that $G=M\times M$ is just the pair groupoid. Then the corresponding index is the classical Atiyah-Singer index $K^0(T^*M)\to K_0(\cK)=\Z$ of (pseudo)differential operators on $M$, \ie the one constructed and computed in \cite{AtSing1}. 
\item Assume that $G$ is the groupoid $M\times_YM$ associated to a smooth fibration $\pi:M\to Y$. The corresponding index is the Atiyah-Singer index  $K^j(T_F^*M)\to K_j(C^*(G))=K^j(Y)$ of families of (pseudo)differential operators on the fibers of $\pi$, \ie the one constructed and computed in \cite{AtSing4}. 
\item  Assume that  $G$ is the groupoid $G=(\widetilde M\times \widetilde M)/\Gamma$ where $\Gamma$ is a countable group acting freely and properly on a manifold $\widetilde M$. The corresponding index   $K^j(T^*M)\to K_0(C^*(G))=K_j(C^*(\Gamma))$. This situation was introduced by Atiyah in \cite{AtiCovSp} where it is shown that the von Neumann dimension of the index (for $j=0$) is in fact the index $K^0(T^*M)\to \Z$. In the $C^*$-context, it was studied in \cite{MischFom} and has been since then studied in many, many papers...
\end{enumerate}

\subsection{Analytic index via deformation groupoids}

The deformation groupoids allow to shade a new light to the pseudodifferential calculus and, in the same time, construct the analytic index without use of pseudodifferential operators. This  was Connes' main motivation for introducing them.

Let $G\rra M$ be a Lie groupoid and let $G_{ad}=\gA\times \{0\}\cup G\times (0,1]$ be its adiabatic groupoid. Consider the evaluation maps $ev_0:C^*(G_{ad})\to C^*(\gA^*)$ and $ev_t:C^*(G_{ad})\to C^*(G)$ for $t\ne0$.

As the sequence $$0\to C^*(G\times (0,1])\lra C^*(G_{ad})\overset{ev_0}{\lra} C^*(\gA^*)\to 0$$ is exact and $C^*(G\times (0,1])$ is contractible, the evaluation $ev_0$ is $K$-invertible. We then have the following important theorem which is in a sense just an observation.

\begin{theorem}[Connes. \cf \cite{ConnesNCG, MonthPie, NWX, DebLescGroupoids}]
The analytic index is the composition $\ind_G=[ev_1]\circ [ev_0]^{-1}$.
\end{theorem}

The proof of this theorem reduces to the following two observations:
\begin{enumerate}
\item (\emph{Naturality of the analytic index.}) Let $G_1\rra M_1$ be a Lie groupoid and let $M_2\subset M_1$ be a closed submanifold which is \emph{saturated} for $G_1$ (\ie for $\gamma \in G_1$ we have $r(\gamma)\in M_2$ if and only if $s(\gamma)\in M_2$). Denote by $G_2\rra M_2$ the Lie groupoid $\{\gamma \in G_1;\ s(\gamma)\in M_2\}$. The algebroid $\gA_2$ of $G_2$ is the restriction to $M_2$ of the algebroid $\gA_1$ of $G_1$. We have restriction maps $r_G:C^*(G_1)\to C^*(G_2)$ and $r_{\gA^*}:C_0(\gA^*_1)\to C_0(\gA^*_2)$. 

Then the diagram \mbox{$\xymatrix{K_j(C_0(\gA^*_1))\ar[r]^{\ind_{G_1}}\ar[d]^{r_{\gA^*}} &K_j(C^*(G_1))\ar[d]^{r_G}\\ K_j(C_0(\gA^*_2))\ar[r]^{\ind_{G_2}}&K_j(C^*(G_2))}$} is commutative.

\item If $E\to B$ is a vector bundle considered as a Lie groupoid, then $$\ind_E:K_j(C_0(E^*))\to K_j(C^*(E))=K_j(C_0(E^*))$$ is the identity.
\end{enumerate}

\begin{remarks}
\begin{enumerate}
\item One can also use the deformation groupoids in order to prove index theorems. In \cite{ConnesNCG}, Alain Connes gives a beautiful proof of the Atiyah-Singer index theorem in $K$-theory - based on the analogue of the Thom isomorphism for crossed products by $\R^n$ of \cite{ConnesThom}. A different proof was given in \cite{DLN}. These proofs naturally generalise to more general index theorems.

\item An extra step is also taken in his lectures at the Coll\`ege de France (see also \eg \cite{ElNaNe}) where the asymptotics of the deformation groupoid are used in order to derive the cohomological formula of the Atiyah-Singer index theorem \cite{AtSing3}. This is a very important issue for index theory, but we will not discuss it any further here.
\end{enumerate}
\end{remarks}

\subsection{Deformation to the normal cone, zooming  action and PDO}\label{sectionDNCtoPDO}

The deformation groupoids as we saw give some insight to the pseudodifferential calculus on a groupoid. We will see in fact that the deformation groupoids allow to recover the pseudodifferential calculus itself, using the natural \enquote{zooming} action (called \emph{gauge action} in \cite{DS1}).

 \subsubsection{The zooming action of $\R_+^*$ on a deformation to the normal cone}\label{actionR}

Let $M$ be a smooth manifold and $V$ a closed submanifold. Denote by $N=\cN_V^M$ the normal bundle of $V$ in $M$. The group $\R_+^*$ acts smoothly on $DNC(M,V)$: for $t\in \R_+^*$ put $\alpha_t(z,\lambda)=(z,t\lambda)$ for $z\in M$ and $\lambda \in \R^*$ and $\alpha _t(x,U,0)=(x,\frac{U}t,0)$ for $x\in V$ and $U\in N_x$.

This zooming action gives two interpretations of conormal distributions.

 \subsubsection{Integrals of smooth functions}

Define first the space $\cS(M,V)$ of Schwartz functions on $DNC_+(M,V)$: notice that $DNC_+(M,V)$ is an open dense subset of the blow-up $SBlup(M\times \R,V\times \{0\})$. The space $\cS(M,V)$ is the set of smooth functions with compact support on $SBlup(M\times \R,V\times \{0\})$ which vanish at infinite order on the complement of $DNC_+(M,V)$. 

\begin{dmc}{The subspace $\cJ(M,V)$}
 An element $k\in \cS(M,V)$ defines a family $(k_t)_{t\ne 0}$ of smooth function on $M$. We define then the subspace $\cJ(M,V)\subset \cS(M,V)$ to be the set of elements $k\in \cS(M,V)$ such that $(k_t)_{t\ne 0}$ vanishes at infinite order at $0$ as a distribution on $M$, \ie such that for every $f\in C^\infty(M)$, the function $t\mapsto \langle k_t|f\rangle=\int _Mk_t(x)f(x)\,dx$ extends to a smooth function on $\R$ which vanishes at infinite order when $t\to 0$.

In local coordinates, \ie if $M$ is a vector bundle over $V$, identifying $DNC(M,V)$ with $M\times \R$, an element $k\in \cS(M,V)$ is in $\cJ(M,V)$ if and only if $\hat k$ vanishes at infinite order on $V\times \{0\}\subset M^*\times \R$. Indeed, in that case, $\langle k_t|f\rangle=\int _{V}\,du\Big(\int_{M^*_u}\hat{k_t}(u,t\xi)\hat f(u,\xi)\,d\xi\Big)$.
\end{dmc}

It is then easy to see - using local coordinates:

\begin{theorem}{\cite{DS1}}
For $m\in \C$, $\cI_c^m(M,V)$ is the set of distributions of the form $\int_{\R^+}k_t\,t^{-1-m}\,dt$ where $k$ runs over $\cJ(M,V)$.
\end{theorem}

\begin{proof}
Me may assume that $M$ is a vector bundle over $V$. We need then to write a symbol $a(u,\xi)\sim \sum a_{m-j}(u,\xi)$ on the dual bundle $M^*$ as an integral $a(u,\xi)=\int _0^{+\infty}g(u,t\xi,t)\,t^{-1-m}\,dt$. Such a $g$ will have a Taylor expansion at $t=0$ of the form $g(u,\xi,t)\sim \sum_{j=0}^{+\infty}t^ng_j(x,\xi)$. Then $$\int _0^{+\infty}g(u,t\xi,t)\,t^{-1-m}\,dt\sim \sum_{j=0}^{+\infty}b_{m-j}(u,\xi)$$ where $b_{m-j}(u,\xi)=\int _0^{+\infty}g_j(u,t\xi)\,t^{n-1-m}\,dt$ is homogeneous in $\xi$ of order $m-j$. Using Borel's theorem, one easily finds a smooth function $g$ whose Taylor expansion $g_j$ yields $b_{m-j}=a_{m-j}$. The theorem follows.
\end{proof}

 \subsubsection{Almost equivariant distributions}
 
 In \cite{vEY1}, Erik van Erp and  Robert Yuncken presented another point of view on pseudo-differential calculus on groupoids. Their construction can be carried to conormal distribution in the following way :
 
 \begin{theorem}
A conormal distribution $P\in \cI_c^m(M,V)$ is a distribution on $M$ with compact support such that there exists a distribution $Q$ on $DNC(M,V)$ given by a smooth family $(P_t)_{t\in \R^*}$, (\ie such that $\langle Q|f\rangle=\int_\R \langle P_t|f_t\rangle\,dt$) which satisfies $P_1=P$ and $\alpha_\lambda Q-\lambda^mQ\in C^\infty (DNC(M,V))$ for every $\lambda\in \R_+^*$.
\end{theorem}

\begin{proof}
 We may of course assume that $M$ is the total space of a vector bundle $E\to V$ (and $V$ is the $0$-section). As $Q$ has compact support, we may write $\langle Q|f\rangle=\int_\R \langle \widehat {P_t}|\widehat{f_t}\rangle\,dt$ where $\widehat{P_t}$ is a smooth function on $E^*$. The family $(\widehat {P_t})_{t\in \R_+}$ is then a smooth function $F$ on $E^*\times \R_+$ such that, for every $\lambda\in \R_+^*$, the function $(x,\xi,t)\mapsto\lambda^mF(x,\xi,t)-F(x,\lambda x,\lambda t)$ has compact support.
 
 If $F$ is of that form, then $(x,\xi)\mapsto F(x,\xi,1)$ is a symbol; if $a$ is a symbol, just put $F(x,\xi,t)=\chi(\|\xi\|^2+t^2)t^{-m}a(x,\tfrac{\xi}{t})$ for $t\ne 0$ and $F(x,\xi,0)=\chi(\|\xi\|^2)a_m(x,\xi)$.
\end{proof}

 \begin{dmc}{Remark: Equivariant linear forms on $\cJ(M,V)$}
One can modify a little bit this construction. The distributions used here are linear forms on $\cS(M,V)$. Such linear forms can be exactly equivariant for $m\in \N$ - and in this case we obtain only differential operators. If instead we take linear forms only defined on the subspace $\cJ(M,V)\subset \cS(M,V)$ of \cite{DS1}, we probably construct an exactly equivariant family of linear forms on $\cJ(M,V)$ extending any element of $\cI_c^m(M,V)$. 
\end{dmc}

\subsection{Some generalizations}

We will just cite here, without going too much into the details, some more general distributions which were constructed and used in operator algebras.

\subsubsection{More general families of pseudodifferential operators}

One can associate useful distributions to much more general symbols. In \cite{HilsSkMorph} were used symbols of type $(\rho,\delta)$ and the associated pseudodifferential operators. 

Let $\rho,\delta\in [0,1]$. Let $E\to M$ be a Euclidean vector bundle. A symbol of order $m$ and type $(\rho,\delta)$ is a function $a:E\to \C$ such that, in local coordinates, for every multiindices $\alpha,\beta$, and every compact subset $K$ in $M$, there exists $C\in \R_+$ such that \[\Big|\frac{\partial^{|\alpha|+|\beta|}a}{\partial^\alpha x\,\partial ^\beta\xi}(x,\xi)\Big|\le C(\|\xi\|+1)^{m-\rho |\beta|+\delta |\alpha|}\]
for every $x\in K$ and $\xi\in E_x$ (see \cite{Hormanderrho}).

Polyhomogeneous symbols, \ie the ones considered above, are particular cases of symbols of type $(1,0)$.

These symbols of type $(\rho,\delta)$ were used in \cite{HilsSkMorph} in order to construct \emph{holonomy almost invariant} transversally elliptic operators on any foliation, \ie holonomy invariant up to lower order. Restricting to a transversal  this amounts to finding operators on a manifold almost invariant under the action of a (pseudo)group $\Gamma$. Thanks to the work of Connes (\cite{ConnesCycCohtransv}), one may assume that the (tangent) bundle $E$ has an invariant subbundle $F$ which has an invariant euclidean metric as well as the quotient $E/F$. In  \cite{HilsSkMorph} were constructed pseudodifferential operators of order $0$ that ``differentiate more'' along the direction $F$ and were used to construct almost invariant Dirac type operators.

\subsubsection{Inhomogeneous calculus}

Connes-Moscovici (\cite{ConnesMosco1, ConnesMosco2}), in order to write formulae in cyclic cohomology used a more specific and unbounded analogue of \cite{HilsSkMorph}. This was a differential operator which is of second order in the direction tangent to $F$ and of order $1$ in the complementary direction. This falls into a construction of an inhomogeneous pseudodifferential calculus modelled on nilpotent groups studied by many authors -  see the book \cite{BealsG}.

\medskip 
In order to understand better this calculus and the corresponding index map, Choi-Ponge (\cite{ChoiPonge1, ChoiPonge2, ChoiPonge3}) and van Erp-Yuncken (\cite{vEY2}) constructed independently a deformation Lie groupoid of the form $(M\times M\times \R^*)\sqcup \cN\times \{0\}$. 

 Let us briefly describe the very general and nice setting for this inhomogeneous calculus. 

Let $M$ be a smooth manifold and let $\{0\}\subset H^1\subset H^2\subset \ldots\subset H^k=TM$ be a filtration of $TM$ by subbundles. We assume that if $X$ is a smooth section of $H^i$ and $Y$ is a smooth section of $H^j$ then $[X,Y]$ is a section of $H^{i+j}$ (we put $H^\ell=TM$ for $\ell\ge k$). We declare then that a vector field which is a section of $H^i$ is a differential operator of degree $i$. 

Note that the bundle $\gN=\bigoplus_{i=1}^nH^i$ is naturally equipped with a nilpotent Lie algebra structure thanks to the Lie brackets of sections: Let $x\in M$. If $X$ is a smooth section of $H^i$ and $Y$ is a smooth section of $H^j$ then the class of $[X,Y]_x$ in $H_x^{i+j}/H^{i+j-1}_x$ only depends on the class of $X_x$ in $H_x^{i}/H^{i-1}_x$ and of $Y_x$ in $H_x^{j}/H^{j-1}_x$. One then has an associated nilpotent Lie group bundle $\cN$ over $M$, which as a set in $\gN$ and the product is constructed thanks to the Baker-Campbell-Hausdorff formula. An important feature of this is that there is a natural action $\alpha$ of the group $\R_+^*$ on $\gN$ and $\cN$ given by $\alpha_\lambda(X)=\lambda^i X$ if $X\in H^i$.

Constructions as those explained in section \ref{sectionDNCtoPDO} (should) naturally allow to recover the associated inhomogeneous pseudodifferential calculus out of this deformation groupoid \cite{vEY2}.

Mohsen \cite{Omar}, gave a very nice construction of this groupoid based on deformations to the normal cone. We will come back to Mohsen's construction of this deformation groupoid in section \ref{sectionOmar}.

\subsubsection{Fourier integral operators}

A ``classical'' family of operators generalising pseudodifferential calculus is that of Fourier integral operators (\cf \cite{Hormander4}). These were constructed by H\"ormander (\cite{Hormander0}) in order to better understand the propagation of singularities for some strictly hyperbolic operators as the wave equation. These operators were studied by several authors and were very useful in local analysis \cite{DuisterHorm, Hormandera, Hormanderb, Hormander4, Duister, Shubin, Treves2}. Recently, an index theory based on Fourier integral operators was developed (\cf \cite{SSS}).

\medskip In \cite{LescVas1}, Lescure and Vassout show how to define Fourier integral operators on Lie groupoids. Fourier integral operators with proper support define multipliers of the convolution algebra $C_c^\infty(G)$, those of order $0$ define multipliers of the $C^*$-algebra of the groupoid, and negative order ones define elements of the $C^*$-algebra. 

\section{Constructions based on Lie groupoids and their deformations}

We briefly discuss here some constructions where deformation groupoids are naturally obtained and used. Such groupoids  give a geometric description of important pseudodifferential calculi. Some others are used to construct elements of Kasparov's $KK$-theory (\cite{Kasparov1980}) such as index maps, or Poincar\'e duality.

\subsection{The associated index map}

Let $G\rra M$ be a Lie groupoid and $\Gamma\rra V$ a sub-Lie groupoid, \ie a submanifold and a subgroupoid of $G$. The groupoid $DNC(G,\Gamma)$ when restricted to the interval $[0,1]$ gives rise to a diagram:
\[\xymatrix{0\ar[r]&C^*(G\times (0,1])\ar[r]&C^*(DNC(G,\Gamma)_{[0,1]}\ar[r]^{\qquad ev_0}\ar[d]_{ev_1}&C^*(\cN_\Gamma^G)\ar[r]\ar@{.>}[dl]^{\partial_\Gamma^G}&0\\
&&C^*(G)}\]
where the top line is exact - at least in the full $C^*$-algebra level.

As $C^*(G\times (0,1])$ is contractible, the map $ev_0$ is invertible in $E$-theory of Connes-Higson (\cite{ConHig}) - and in $KK$-theory if this exact sequence admits a completely positif splitting - which is the case if the groupoid $\cN_\Gamma^G$ is amenable. 

We thus obtain an index element $\partial_\Gamma^G=[ev_1]\otimes [ev_0]^{-1}\in E(C^*(\cN_\Gamma^G),C^*(G))$.

Let us see some examples:

\subsubsection{The \enquote{Dirac element} of a Lie group}

Consider an inclusion $H\subset G$ of Lie groups. Note that the groupoid $\cN_H^G$ is actually a group. This group is immediately seen to be the semidirect product $H\ltimes (\gG/\gH)$ - where $H$ acts on the lie algebra $\gG$ of $G$ via the adjoint representation of $G$ and fixes the Lie algebra $\gH$ of $H$.

Assume $G$ is a (almost) connected Lie group and $K$ is its maximal compact subgroup. The $K$ theory of the group $C^*(\cN_K^G)$ is a twisted $K$ theory of the group $K$  and the map $\partial_K^G:K_0(C^*(\cN_K^G))\to C^*_r(G)$ identifies with the \enquote{Dirac element} -- \ie the Connes-Kasparov map.

\subsubsection{Foliation and shriek map for immersions}

Let $(M_1,F_1)$ and $(M_2,F_2)$ be smooth (regular) foliations. In \cite{HilsSkMorph} is considered a notion of maps between leaf spaces $f:M_1/F_1\to M_2/F_2$. The goal of that paper is to construct wrong way functoriality maps $f!:K(C^*(M_1,F_1))\to K(C^*(M_2,F_2))$ generalizing  constructions of \cite{CoSk}.

As in \cite{CoSk}, writing $f$ as a composition $p\circ i$ where $i:M_1/F_1\to M_1/F_1\times M_2/F_2$ is (somewhat loosely speaking) $\ell\mapsto (\ell,f(\ell))$ - where $\ell$ is a leaf of $(M_1,F_1)$ and $p$ is the projection $M_1/F_1\times M_2/F_2\to M_2/F_2$ the problem is reduced to the case of immersions and submersions.

\medskip 
Following Connes' construction of the tangent groupoid, a deformation groupoid was used in \cite{HilsSkMorph} in order to construct the wrong way functoriality map for immersions between spaces of leaves, in the following way:
\begin{itemize} \item Using a Morita equivalence, one may reduce to transversals in order to understand an immersion $M_1/F_1\to M_2/F_2$ to be an inclusion $f:G_1\hookrightarrow G_2$ where $M_1$ is a submanifold of $M_2$ which is saturated and $G_1$ is the restriction of $G_2$ to $M_1$ - \ie for $\gamma\in G_2$ we have the equivalences: $s(\gamma)\in M_1\iff r(\gamma)\in M_1\iff \gamma\in G_1$. 

\item Then, the DNC construction is used in order to obtain a wrong way functoriality element $f!\in E(C^*(G_1),C^*(G_2))$
(\footnote{In  \cite{HilsSkMorph} this element is just a morphism of $K$-groups since $E$-theory of Connes-Higson was defined later}). This element is the Kasparov product $[th]\otimes \partial _{G_1}^{G_2}$ of a Thom isomorphism element $[th]\in KK(C^*(G_1),C^*(\cN_{G_1}^{G_2}))$ with the index element $\partial _{G_1}^{G_2}\in E(C^*(\cN_{G_1}^{G_2}),C^*(G_2))$. 
\end{itemize}

\begin{remark}
In order to construct the Thom element $[th]$ in  \cite{HilsSkMorph} one of course has to assume a $K$-orientation. Moreover, is used the fact that the groupoid $G_1$ acts naturally on the normal bundle $\cN_{M_1}^{M_2}$.
\end{remark}

\begin{question}
Can one construct the Thom element  when the groupoid $G_1$ does not act on the bundle $\cN_{M_1}^{M_2}$? What is the right condition of $K$-orientation for $\cN_{G_1}^{G_2}$?
\end{question}

\begin{remark} It is mentioned also in \cite{HilsSkMorph} that one could use a deformation groupoid to construct $f!$ for submersions.
\end{remark}

\subsubsection{On the computation of the index map in some cases}

In \cite{DS4},  the index map  index element $\partial_\Gamma^G=[ev_1]\otimes [ev_0]^{-1}\in E(C^*(\cN_\Gamma^G),C^*(G))$ associated to the inclusion of groupoids is computed in some situations.
In particular, when $\Gamma $ is just a space $V\subset M$,  the $C^*$-algebra of the groupoid $\cN_V^G$ has the same $K$ theory as the space $\cN_V^G$. We have an embedding $j:\cN_V^G\to \gA G$ of this space, via a tubular neighborhood construction, as an open subset of the Lie algebroid of $G$. The index $\partial_\Gamma^G$ is the composition:
$$K_*(C^*(\cN_V^G))\simeq K_*(C_0(\cN_V^G))\overset {j}{\lra}K_*(C_0(\gA G))\overset {\ind _G}{\lra}K_0(C^*(G))$$
(and this is actually true for $KK$-elements instead of morphisms of $K$-theory).

\bigskip
Also, one can compare it to the connecting map $\tilde \partial_\Gamma^G$ of the exact sequence $$0\to C^*(\mathring G)\to C^*(SBlup(G,\Gamma))\to C^*(S\cN_\Gamma^G)\to 0,$$
associated with the open saturated subset $\mathring{M}=M\setminus V\subset SBlup(M,V)$. Here, $\mathring{G}$ is the groupoid $G_{\mathring{M}}^{\mathring{M}}$.

We then have a commutative diagram (\footnote{The arrows are in fact $E$-theory elements or even $KK$-theory elements.}) $$\xymatrix{K_j(C^*(\mathring G))\ar[d]_j \ar[r]^{\tilde \partial _\Gamma^G\ \ }&K_{j+1}(C^*(S\cN_\Gamma^G))\ar[d]^{th}\\
K_j(C^*(G))\ar[r]^{\partial _\Gamma^G}&K_{j}(C^*(\cN_\Gamma^G))}$$ where $th$ is a map based on the Connes analogue of the Thom isomorphism (\cf \cite{ConnesThom}, see also \cite{FaSkThom} for its construction in Kasparov's bivariant groups). 

It is then easily seen that:\begin{itemize}
\item if no $G$ orbit is contained in $V$, then the inclusion $C^*(\mathring{G})\to C^*(G)$ is a Morita equivalence and therefore $j:K_*(C^*(\mathring{G}))\to K_*(C^*(G))$ is an isomorphism;
\item if for every $x\in V$, the tangent to the $G$ orbit through $x$, \ie the image by the anchor map $\varrho_x:\gA G_x\to T_xM$ is not contained in $T_xV$, then the Thom element $th$ is also an isomorphism.
\end{itemize}

\subsubsection{Full index}

Another natural question, involving blowup  groupoids, appears: Let $P$ be an elliptic operator on $SBlup_{r,s}(G,\Gamma)$. When is it invertible modulo $C^*(\mathring{G})$? If this is the case, can one compute its index as an element in $K_*(C^*(\mathring{G}))$? There is a particular interest when $\mathring{G}=\mathring M\times \mathring{M}$, in which case the index is in $\Z$.

We have a commutative diagram
$$\xymatrix{
&0\ar[d]&0\ar[d]&0\ar[d]\\
0\ar[r] &C^*(\mathring{G})\ar[r]\ar[d]& \Psi^*(\mathring{G})\ar[r]\ar[d] &C_0(\bS^*\gA \mathring{G})\ar[r]\ar[d]& 0\\
0\ar[r] &C^*(SBlup_{r,s}(G,\Gamma))\ar[r]\ar[d] & \Psi^*(SBlup_{r,s}(G,\Gamma))\ar[r]^\sigma\ar[d]^q &C_0(\bS^*\gA SBlup_{r,s}(G,\Gamma))\ar[d]\ar[r]& 0\\
0\ar[r] &C^*(S\cN _\Gamma^G)\ar[r]\ar[d]& \Psi^*(S\cN _\Gamma^G)\ar[r]\ar[d] &C_0(\bS^*\gA S\cN_\Gamma^G)\ar[r]\ar[d]& 0\\
&0&0&0}
$$
where lines and columns are exact. \begin{itemize}
\item The columns represent the exact sequence corresponding to the partition of the groupoid $SBlup_{r,s}(G,\Gamma)$ into the open subgroupoid $\mathring{G}$ and the closed subgroupoid $S\cN_\Gamma^G$ at the level of $C^*$-algebras, order $0$ pseudodifferential operators and principal - $0$-homogeneous symbol. 
\item The lines are the exact sequences of zero order pseudodifferential operators of the groupoids $\mathring{G}$,  $SBlup_{r,s}(G,\Gamma)$ and $S\cN_\Gamma^G$ (see theorem \ref{thmExacSeq}).
\end{itemize}
It follows that a pseudodifferential operator \ie an element  $P\in \Psi^*(SBlup_{r,s}(G,\Gamma))$ is invertible modulo $C^*(\mathring{G})$ if and only if its \emph{classical symbol} $\sigma(P)$ and its \emph{non commutative symbol}, \ie its restriction $q(P)$ to the \enquote{singular part} $S\cN_\Gamma^G$ are both invertible.

\medskip In other words, we have a \emph{full symbol algebra} $\Sigma_\Gamma^G=C_0(\bS^*\gA SBlup_{r,s}(G,\Gamma))\times_{C_0(\bS^*\gA S\cN_\Gamma^G)}\Psi^*(S\cN _\Gamma^G)$ and an exact sequence $$0\to C^*(\mathring{G})\to \Psi^*(SBlup_{r,s}(G,\Gamma))\to \Sigma_\Gamma^G\to 0.$$
One can compute in some cases the $K$-theory of $\Sigma_\Gamma^G$ and the connecting map. For instance, if - as above $\Gamma$ is just a submanifold $V\subset M$ and if we assume that the anchor map $\varrho_x:\gA G_x\to T_xM$ is not contained in $T_xV$, then $\Sigma_V^G$ is $K$-equivalent to the algebra of pseudodifferential operators on $G$ whose symbol is \enquote{trivial} when restricted to $V$ -- \ie a function on $V$, and the index map is the restriction to this subalgebra of $\Psi^*(G)$ to the index map of $G$.

\subsection{Groupoids using deformation constructions}

Lie groupoids are useful in defining natural pseudodifferential calculi. In several $K$-theoretic constructions, many authors have introduced interesting groupoids using various techniques: gluing, integration of algebroids... Recognizing some of these groupoids as deformations or blowups often simplifies their construction, may help their understanding and give a geometric insight on their properties.

In this section we outline some of these natural and useful constructions of deformation or blowup Lie groupoids.

\subsubsection{Pseudodifferential calculi on singular manifolds}

Let $M$ be a singular manifold. This can be a manifold with boundary, or with corners, or even a stratified manifold. Many natural pseudodifferential calculi were constructed by analysts -- especially in the school of Richard Melrose -- in order to take into account, sometimes in a very fine way, the behavior of the operators near the boundary (\cf \cite{Melbook, MM, Mazzeo91, MM2, MelPia}). 

Some of these calculi were already constructed using blowup constructions. Some others just putting some conditions on the riemannian metric in the regular part $\mathring{M}\subset M$ degenerating near the boundary. In many cases, when this metric is complete, this metric actually corresponds to an algebroid. More precisely, the space of bounded vector fields with respect to this metric is the module of sections of a Lie algebroid on $M$. It follows from \cite{Debord} that this algebroid integrates to a Lie groupoid $G\rra M$. 

Of course knowing this groupoid will certainly not solve at once all the questions for which the corresponding calculus was constructed! It may however help understanding some of its properties: the decomposition of $C^*(G)$ into ideals that can often be seen geometrically can simplify the study of conditions for Fredholmness of an associated (pseudo)differential operator; it is  also relevant for various index computations...

\bigskip Let us outline specific examples.

\paragraph{The groupoid of the $b$-calculus.} Let $M$ be manifold with boundary. Melrose constructs the $b$ space which is the blow up of $M\times M$ along its corner $\partial M\times \partial M$ (\cf \cite{Mel2, Mel1, MelPia, Melbook}). The pseudodifferential operators of the corresponding $b$-calculus consists of operators which are distributions on the $b$-space that are conormal along the diagonal $M$ and have a specific decay near the boundary components $M\times \partial M$ and $\partial M\times M$.

\medskip Monthubert (\cite{Month1, Month2}) constructed the associated $b$-groupoid which is nothing else than the dense open subspace $G_b=SBlup_{r,s}(M\times M,\partial M\times \partial M)$ of the $b$-space $SBlup(M\times M,\partial M\times \partial M)$.

\medskip Let us mention that in fact all these constructions were also performed in the more general case of manifolds with corners.

\paragraph{Fibered corners.}  We restrict again to the case of a manifold with boundary, although the constructions below extend to more general settings of manifolds with corners. 

Let $M$ be manifold with boundary $\partial M$ and let $p:\partial M\to B$ be a fibration. Mazzeo (\cite{Mazzeo91}) studied the \emph{edge calculus} in this situation. This corresponds to the blowup $SBlup(M\times M,\partial M\times_B \partial M)$ and of course to the corresponding groupoid $G_e=SBlup_{r,s}(M\times M,\partial M\times _B\partial M)$.

Later, (\cite{MM2}) Mazzeo and Melrose introduced and studied in the same situation the $\Phi$ calculus which corresponds to the algebroid of vector fields that are tangent to the fibers at the boundary but also whose derivative is tangent to $\partial M$: these are vector fields of the form $X+tY+t^2N$ - where $t$ is a defining function of the boundary, $X$ is a vector field along the fibration (extended near the boundary), $Y$ is tangent to the boundary and $N$ is normal to the boundary. Piazza and Zenobi (\cite{PiaZen}) actually realized that the groupoid constructed in \cite{DLR} integrating this algebroid can be obtained via a double blowup construction. See also \cite{Yama} for topological aspects of indices in this context.

\begin{question} \label{Question2}
A natural question is also to try to understand the non complete case too in terms of deformation groupoids...
\end{question}

\subsubsection{Inhomogeneous pseudodifferential calculus}
\label{sectionOmar}

The ``inhomogeneous'' pseudodifferential calculus (also called ``filtered'' or ``Carnot''), deals with manifolds $M$ whose tangent bundle is endowed with a filtration $(H_i)_{0\leqslant 1\leqslant k}$ (with $H_0=0$ and $H_k=TM$) satisfying $[\Gamma(H_i),\Gamma(H_j)]\subset \Gamma(H_{i+j})]$. A natural pseudodifferential calculus has been constructed in this framework, generalizing the case of contact manifolds (\cf \cite{BealsG}). 
In this sub-elliptic or Carnot calculus, the vector fields that are sections of the bundle $ H_i $ are considered as differential operators of order $i$. 

A deformation groupoid taking into account this inhomogeneous calculus was constructed independently in (\cite{Ponge1, ChoiPonge1, ChoiPonge2, ChoiPonge3}) and \cite{VanErpAS1, JuvE, vEY2}.   Both these constructions are rather technical and based on higher jets. This groupoid plays the role of Connes' tangent groupoid in this setting. It allows to recover the inhomogeneous pseudodifferential calculus (\cf \cite{vEY2}).

Omar Mohsen presents in \cite{Omar} a very elegant construction of this deformation groupoid in terms of successive deformations to the normal cone. In the case where there is only one sub-bundle $H\subset TM$, one just considers the inclusion of $H\times\{0\}$ into Connes' tangent groupoid  $(M\times M)\times \R^*\sqcup TM\times \{0\}$.  The crucial fact that the object built is canonically a groupoid is clear in this construction - while it leads to relatively sophisticated computations in the works cited above. The general case is treated by induction.

In addition, this construction has the advantage of being very flexible and generalizing immediately, for example in the context of an inhomogeneous a pseudodifferential calculus transverse to a foliation - as the one appearing in the work of Connes-Moscovici (\cite{ConnesMosco1, ConnesMosco2}).

\subsubsection{Poincar\'e dual of a stratified manifold}

\paragraph{K-duality.} G.G. Kasparov in \cite{Kasparov1980} defines a formal $KK$-duality of $C^*$-algebras. If $A$ and $B$ are $K$-dual $C^*$-algebras, the $K$-homology of $A$ is isomorphic to the $K$-theory of $B$. 

\medskip Let $M$ be a smooth compact manifold. The algebra $C(M)$ has naturally a $K$-dual which is $C_0(T^*M)$ (\cf \cite{Kasparov1975, Kasparov1980, CoSk}).  The corresponding duality map associates to the $K$-homology class of an elliptic (pseudo)differential the $K$-theory class of its symbol.


Several generalisations to manifolds with singularities have been given by various authors: manifolds with boundary (\cite{CoSk}), non Hausdorff manifolds (\cite{KaSk1}), manifolds with conic singularity (\cite{DebLesc, CRLM}), stratified manifolds (\cite{DebLesc2}). Many of them use naturally Lie groupoids.

\paragraph{Manifolds with a conic singularity.} Let us outline the construction of \cite{DebLesc}.

Let $M$ be a compact manifold with boundary $\partial M$. Denote by $\mathring{M}$ the open subset $M\setminus \partial M$. The one point compactification ${M}_+$ of $\mathring{M}$ is the quotient of $M$ by the equivalence relation which identifies all the points of the boundary $\partial M$. It is a \emph{manifold with conic singularity}.

Let $G_b=Blup_{r,s}(M\times M,\partial M\times \partial M)=\mathring{M}\times \mathring{M}\sqcup \partial M\times \partial M\times \R_+^*$ be the groupoid of the $b$-calculus of the manifold with boundary $M$. The Poincar\'e dual of $M_+$ constructed in  \cite{DebLesc} is a closed subgroupoid $\cG$ of the ``adiabatic'' groupoid $DNC(G_b,G_b^{(0)})$ of $G_b$: it is $\cG=DNC(G_b,G_b^{(0)})\setminus \mathring{M}\times \mathring{M}\times \R_+^*$ \cite{DebLesc}. Note that $\cG$ is the union of the algebroid $\gA G_b=TM=T\mathring M\sqcup T\partial M\times \R$ of $G_b$ with $\partial M\times \partial M\times \R^2$.

\emph The \emph{Poincar\'e duality} element is an element $\psi\in KK(C(M_+)\otimes C^*(\cG),\C)$. It is obtained as follows: \begin{itemize}
\item first $C(M_+)$ sits naturally in the center of the multiplier algebra of $C^*(\cG)$, by extending in a unital way the map $C_0(\mathring M)\to \cM(C_0(T^*\mathring M))=\cM(C^*(TM))$. We thus have a morphism of $C^*$-algebras $m:C(M_+)\otimes C^*(\cG)\to C^*(\cG)$.
\item As $C^*(\mathring{M}\times \mathring{M}\times \R_+^*)\simeq \cK\otimes \C_0(\R_+^*)$, the extension $$0\to C^*(\mathring{M}\times \mathring{M}\times \R_+^*)\to C^*(DNC(G_b,G_b^{(0)}))\to C^*(\cG)\to 0$$ gives rise to an element $d\in KK^1(C^*(\cG),C_0(\R_+^*))=KK(C^*(\cG),\C)$.
\end{itemize}
Put then $\psi=m^*(d)$.

\begin{question}\label{Question3}
It would be nice to understand this Poincar\'e duality by describing as precisely as possible which operators on $\mathring{M}_+$ correspond to symbols on $C^*(\cG)$. In particular, what is the symbol class of a Fuchs type operator studied in \cite{Lescure1}? This question is certainly linked with question \ref{Question2}.
\end{question}

\paragraph{Stratified manifolds.} One can generalize immediately this construction to manifolds with a fibered boundary: given a fibration $p:\partial M\to B$ one can form the space $M/\sim$,  where $\sim $ is the equivalence relation on $M$ given by $$x\sim y\iff \begin{cases} x=y\ \ $or$\\
x,y\in \partial M \ $and$\ p(x)=p(y).\end{cases}$$ 
One just replaces the groupoid $G_b$ by the groupoid $G_e$ of the edge calculus or by the groupoid $G_\Phi$  of the  $\Phi $ calculus. 

This construction was extended in \cite{DebLesc2}, using an induction process, to describe in a similar way the Poincar\'e dual of general stratified manifolds $X$. The construction can in fact be obtained by use of several blowups: one blows up inductively all the strata to obtain a groupoid $G_X$ generalizing $G_b$; then, one uses as above the adiabatic deformation of $G_X$ in order to construct the dual groupoid $\cG_X=DNC(G_X,G_X^{(0)})\setminus \mathring{X}\times  \mathring{X}\times \R_+^*$ and the Poincar\'e duality element $\psi_X\in KK(C(X)\otimes C^*(\cG_X),\C)$. Details will appear in \cite{DebordPrep}.

\section{Related topics and further questions}

In this section, we examine some topics that are not a priori based on Lie groupoids, but  are actually linked to our discussion.  

\subsection{Relation to Roe algebras}

Starting from a finite propagation speed principle for differential operators of order one \cite{Chernoff}, John Roe developed a very beautiful theory of \emph{coarse} spaces and algebras \cite{Roebook}. The $K$-theory of Roe algebras have been used as a receptacle for various index problems \cite{HigRoe}. 

We will not develop here this theory. Let us just outline some links with groupoids that have been made in the literature.
\begin{itemize}
\item In \cite{STY}, a groupoid is constructed out of a coarse space and it is shown that the Roe $C^*$-algebra of locally compact functions with finite propagation on this coarse space is the $C^*$-algebra of this groupoid
\item Out of a Lie groupoid one constructs Roe type algebras (see \eg \cite{BenRoy}). A natural example is the case of a covering space $\widetilde M\to M$ with group $\Gamma$. In that case, one may also define naturally the index with values in the $K$-theory $C^*(\Gamma)$ for $\Gamma$-invariant elliptic operators on $\widetilde M$ using coarse techniques. In particular, there is an exact sequence of Roe algebras (\cite{HigRoeI, HigRoeII, HigRoeIII, PiaSch}) $$0\to \cC^*(\widetilde M)^\Gamma\longrightarrow \cD^*(\widetilde M)^\Gamma \longrightarrow \cD^*(\widetilde M)^\Gamma/\cC^*(\widetilde M)^\Gamma\to 0.$$
The algebra $\cC^*(\widetilde M)^\Gamma$ is Morita equivalent to $C^*(\Gamma)$ and that of $ \cD^*(\widetilde M)^\Gamma/\cC^*(\widetilde M)^\Gamma$ is the $K$-homology of $M$ (Paschke duality -- \cite{Paschke}). 

\item In this precise case, Zenobi (\cite{Vito}) identified the $K$-theory exact sequence of the Higson-Roe sequence with the adiabatic groupoid exact sequence of the groupoid $(\widetilde M\times \widetilde M)/\Gamma \rra M$.  
\end{itemize}

\begin{question}
How far can one push this parallel between groupoid $C^*$-algebras and Roe algebras?
\end{question}

\subsection{Singular foliations and \enquote{singular Lie groupoids}}

Let us just say a few words on ``very singular Lie groupoids'' associated with singular foliations in \cite{AndrSk1}. 

A singular foliation on a compact manifold $M$ is a  submodule $\cF$ of the $C^\infty(M)$ module of vector fields $\Gamma(TM)$ which is finitely generated and \emph{involutive} -- \ie closed under Lie brackets: $[\cF,\cF]\subset \cF$.

The module $\cF$ tells us which differential operators are ``longitudinal''. In  \cite{AndrSk1} is constructed the holonomy groupoid and the $C^*$-algebra of such a singular foliation. In   \cite{AndrSk2} is also constructed a pseudodifferential calculus and an analytic index for a singular foliation as well as a deformation holonomy groupoid, which gives rise to an alternate way of defining the analytic index.

\begin{question}
Let $(M,\cF)$ be a singular foliation in the sense of  \cite{AndrSk1}. It is sometimes possible by blowing up singular leaves to obtain a less singular foliation: one for which the holonomy groupoid is a Lie groupoid. 

For instance, the singular foliation of $\R^3$ whose leaves are the spheres of center $0$ -- given by the action of $SO(3)$ -- has a singular holonomy groupoid $\{(x,y)\in \R^3\setminus \{0\};\ \|x\|=\|y\|\}\sqcup \{0\}\times SO(3)$. Blowing up the singular leaf $\{0\}$, we obtain the regular foliation on $\R_+\times \bS^2$ with leaves $\{r\}\times \bS^2$.

This poses several questions

\begin{enumerate}
\item How general/canonical can this procedure be? 

\item When such a blowup is possible, what is the precise relation between the $C^*$-algebra of the foliation in the sense of \cite{AndrSk1} and that of the Lie groupoid obtained by blowing up?

\end{enumerate}
\end{question}

\begin{question}
The holonomy groupoid $Hol(M,\cF)$ of a singular foliation is a ``bad'' topological space. On the other hand, there is a natural notion of a smooth map $V\to Hol(M,\cF)$ -- and also of a smooth submersion -- for a manifold $V$ (this notion is called a \emph{bisubmersion} in \cite{AndrSk1}). Is there a natural structure to express this in a functorial way?
\end{question}


\subsection{Computations using cyclic cohomology }

A classical way to make $K$-theoretic computations for spaces is to use the Chern isomorphism with (co)homology. In the case of manifolds, one naturally uses the de Rham cohomology. The de Rham cohomology extends to the noncommutative setting thanks to Connes' cyclic cohomology. On the other hand, de Rham cohomology uses differentiation and is not well suited for continuous functions, but rather smooth functions. So, starting with Connes (\cf \cite{ConnesCycCohtransv, ConnesNCG}) and then many others, one constructs natural cyclic cocycles defined on the algebra $C_c^\infty(G)$ of smooth functions with compact support on a Lie groupoid $G$. Then one encounters a difficult question: extend these cyclic cocycles to a larger algebra that has the same $K$-theory as $C^*(G)$. Such an extension was performed in \cite{ConnesCycCohtransv} using a notion of $n$-traces - which are \enquote{well behaved cyclic cocycles}.

\begin{question}
Extend natural cocycles so that they pair with the $K$-theory of the $C^*$-algebra of the groupoid and compute this pairing.
\end{question}

\subsection{Behaviour of the resolvent of an elliptic operator}

Let us start with say a positive laplacian $\Delta$ on a groupoid $G$ with compact $G^{(0)}$. As $\Delta $ is elliptic, self adjoint and positive, the operator $1+\Delta $ is invertible and its inverse is in $C^*(G)$ (\cf \cite{Vassout}). 

\begin{question}
What is the behavior of $(1+\Delta)^{-1}$? Is it a smooth function outside $G^{(0)}$? What kind of decay at infinity does it have? 
\end{question}

Just a few cases have been worked out (\cf \cite{Melbook, MM}). 

\smallskip Let us make a comment about these problems. The elements of $C^*_r(G)$ define distributions on $G$. In other words there is an injective map $C^*_r(G)\to C^{-\infty}(G)$. So it is a natural question to ask what kind of distributions are elements like $(1+\Delta)^{-1}$. The distribution associated with an element of $C^*(G)$ only depends on its image in $C^*_r(G)$. So that these problems a priori concern the reduced $C^*$-algebra of $G$.

%

\subsection{Relations with the Boutet de Monvel calculus}

This was a motivation for us from the beginning. 

Let us very briefly say a few words on the Boutet de Monvel calculus. Details can be found in \cite{MB1, MB2, Grubb, Schrohe}.

Let $M$ be a manifold with boundary. We consider $M$ as included in a manifold $\widetilde M$ without boundary in which a smooth hypersurface $\partial M$ of $\widetilde M$ separates $\widetilde M$ into two open subsets $\mathring M$ and $M_-$.

\medskip Denote by $\chi_M$ the characteristic function of $M$. Boutet de Monvel defines:

\begin{definition}
 A pseudodifferential operator $\Phi$ (with compact support) on $\widetilde M$ is said to have the \emph{transmission property} if for every smooth function $\tilde f$ on $\widetilde M$,  then $\Phi(\chi_M \tilde f)$ coincides on $\ronde M$ with a smooth function on $\widetilde M$ (\footnote{We actually have to assume that the same holds also for the adjoint of $\Phi$.}).
\end{definition}

Of course, a smoothing operator satisfies the transmission property, and thus this property can be described in terms of the (restriction to $\partial M$ of the) total symbol of $\Phi$. This condition was explicitly computed in local coordinates (see \cite{MB1, MB2, Grubb, Schrohe}).

\medskip 
Assume that $\Phi$ satisfies the transmission property, and let $\Phi_+(f)$ be the restriction to $M$ of the smooth function which coincides with $\Phi(\chi_M \tilde f)$ on $\ronde M$.

\medskip 
The  operators $\Phi_+$, \emph{do not form an algebra} since $\Phi_+\Psi_+\ne (\Phi\Psi)_+$, and
the difference $\Phi_+\Psi_+- (\Phi\Psi)_+$ is not a pseudodifferential operator. On the other hand, this difference belongs to a new class of operators -- described again precisely in local coordinates -- called \emph{singular Green} operators.

\medskip
The set of operators of the form $\Phi_++S$ where $P$ is a pseudodifferential operator with the transmission property and $S$ a singular Green operator is an algebra. Call it $\cP_{BM}(M)$.

\medskip Boutet de Monvel moreover defines \emph{singular Poisson} (or \emph{Potential}) operators mapping functions on $\partial M$ to functions on $M$ and  \emph{singular Trace} operators which map functions on $M$ to functions on $\partial M$. Singular Poisson operators and singular Trace operators are adjoint of each other. 

\bigskip They form bimodules yielding  a Morita equivalence  between singular Green operators on $M$ and ordinary pseudodifferential operators on its boundary.

We thus obtain the \emph{Boutet de Monvel algebra} which consists of matrices of the form $\begin{pmatrix} \Phi_++S&P\\T&Q
\end{pmatrix}$  where \begin{itemize}
\item $\Phi$ is a pseudodifferential operator on $\widetilde M$ with the transmission property, and $\Phi_+$ the corresponding operator on smooth functions on $M$;
\item $S$ is a singular Green operator acting on $M$;
\item $P$ is a singular Poisson operator mapping functions on $\partial M$ to functions on $M$;
\item $T$  is a singular trace operator mapping functions on $M$ to functions on $\partial M$;
\item $Q$ is a pseudodifferential operator on $\partial M$.
\end{itemize}

All these constructions are generalized to the case of families of manifolds and more generally to a Lie groupoid $G\rra M$ on a manifold with boundary $M$ assuming that $G$ is \emph{transverse} to the boundary $M$ (\cf \cite{MeloSchSch, LoMelo, Bohlen2, Bohlen1}).

\bigskip When taking the closure of the \emph{bounded} singular Green operators, we find an exact sequence $$0\to \cK(L^2(M))\longrightarrow {\mathrm{Green}}\overset{\sigma_g}{\longrightarrow} \Sigma^G_M\to0,\eqno{(\mathrm{Green})}$$ where $\sigma_M^G$ is the \emph{non commutative symbol} of singular Green operators with values in the algebra $\Sigma_M=C_0(S^*\partial M)\otimes \cK$. This exact sequence can be compared with the exact sequence $$0\to \cK(L^2(\partial M))\longrightarrow \Psi^*(\partial M)\overset{\sigma_{\partial M}}{\longrightarrow} C_0(S^*\partial M)\to0,\eqno{(\Psi_{\partial M})}$$ of pseudodifferential operators on $\partial M$. Bounded singular Poisson operators and  singular trace operators form bimodules yielding a Morita equivalence of these sequences. 

\medskip
Now, the group $\R_+^*$ naturally acts on Connes tangent groupoid $T\partial M\times \{0\}\sqcup (\partial M\times \partial M)\times \R_+^*$ of the manifold  $\partial M$ via the zooming action. The corresponding crossed product groupoid $\cG$ gives naturally rise to an exact sequence  $$0\to \cK(L^2(\partial M)\times \R_+^*))\longrightarrow C^*(\cG)\overset{\ev_0}{\longrightarrow} C_0(T^*\partial M)\rtimes \R_+^*\to0.\eqno{(\cG)}$$
Note that $C_0(S^*\partial M)\otimes \cK$ naturally sits in $C_0(T^*\partial M)\rtimes \R_+^*$ as an ideal. In \cite{AMMS}, the exact sequence $(\cG)$ restricted to this ideal, was shown to  coincide with (Green) - by showing that they define the same $KK$-element and then using Voiculescu's theorem \cite{Voicu}.

This construction was generalized in  \cite{DS1}, where for any Lie groupoid $G\rra V$, is directly constructed a (sub)-Morita equivalence relating the pseudo-differential exact sequence $$0\to C^*(G)\longrightarrow \Psi^*(G)\overset{\sigma_{G}}{\longrightarrow} C_0(S^*\gA G)\to0,\eqno{(\Psi_{G})}$$  of the groupoid $G$ and the exact sequence $$0\to C^*(G)\otimes \cK(\L^2(\R_+^*))\longrightarrow C^*(\cG_G)\overset{\ev_0}{\longrightarrow} C_0(\gA^*G)\rtimes \R_+^*\to0,\eqno{(\cG_G)}$$ where $\cG_G$ is the \enquote{gauge adiabatic groupoid} obtained as the crossed product by the natural scaling action of $\R_+^*$ on the adiabatic groupoid $G_{ad}=\gA G\times \{0\}\sqcup G\times \R_+^*$.

In \cite{DS5}, the construction of the  \enquote{gauge adiabatic groupoid} is generalized using the blowup construction for groupoids discussed above (\cf section \ref{sec:blowup}).  If $V\subset M$ is a hypersurface which is transverse to the groupoid $G$, the groupoid  $Blup_{r,s}(G,V)\to Blup(M,V)\simeq M$ is the gauge adiabatic groupoid of $G_V^V$. More generally, if $V$ is \emph{any} submanifold of $M$ which is transverse to $G$, one still constructs a Boutet de Monvel type calculus.

\begin{question}
It is natural to try to show that the closure of the algebra of bounded elements of the Boutet de Monvel algebra coincides with the one obtained in this way. We have quite well understood this and should write it precisely. Actually, the transmission property gives rise to a small difference between them.
\end{question}

\subsection{Algebroids and integrability}

Recall from definition \ref{algebroid} that an algebroid over a manifold $M$ is a smooth bundle $A$ over $M$ endowed with the following structure\begin{itemize}
\item  a Lie algebra bracket on the space of smooth sections of $A$;
\item a bundle morphism $\varrho :A\to TM$.
\end{itemize}
These are assumed to satisfy: $[X,fY]=\varrho (X)(f) Y + f[X,Y]$.

The integrability problem for algebroids is not ``trivial''. There are indeed algebroids that are not associated with groupoids (\cf \cite{AlMoli} -- se also \cite{CraFer} where necessary and sufficient conditions for integrability are given). 

On the other hand, one can define differential operators on an algebroid and even pseudodifferential operators locally using local integration, which is also possible. These operators act naturally on $C^\infty(M)$ (or $L^2(M)$ when they are bounded). When the algebroid is integrated to a Lie groupoid $G$, the $C^*$-algebra $C^*(G)$  is generated by resolvants of elliptic operators. Also, in that case, we find more representations where these operators act...

\begin{question}
How much of a groupoid $C^*$-algebra can one construct out of just its algebroid?  Is there a construction of an ($s$-simply connected) Lie groupoid $C^*$-algebra which makes sense even for non integrable algebroids?
\end{question}

\begin{remark}
Let us remark that we can also restrict this question to the case of Poisson manifolds which are particular cases of Lie algebroids.
\end{remark}

\subsection{Fourier integral operators}

Fourier integral operators (\cite{Hormandera, Hormanderb, DuisterHorm}) form a class of operators, larger than pseudodifferential operators. They are very useful in order to understand hyperbolic differential operators. 

In \cite{LescVas1}, Fourier integral operators on a Lie groupoid are defined and studied. We do not wish to say much on these. Let us just ask a question.

\begin{question}
Is there a construction of Fourier integral operators in the spirit of \cite{DS1} or \cite{vEY1} (See section \ref{sectionDNCtoPDO} above) ? 
\end{question}

\bibliography{Biblioreview.bib} 
\bibliographystyle{amsplain}

\end{document}